\documentclass[11pt,reqno]{amsart}
\usepackage[letterpaper, portrait, margin=1in, headheight=1in]{geometry}
\usepackage[T1]{fontenc}
\usepackage{mathtools}
\usepackage{enumerate,nicefrac}
\usepackage{amsmath,amsthm,amssymb,amsfonts,shuffle,cases,centernot}
\usepackage{hyperref}
\hypersetup{
    unicode=false,          % non-Latin characters in AcrobatÔøΩs bookmarks
    pdftoolbar=true,        % show AcrobatÔøΩs toolbar?
   pdfmenubar=true,        % show AcrobatÔøΩs menu?
   pdffitwindow=false,     % window fit to page when opened
   pdfstartview={FitH},    % fits the width of the page to the window
   pdftitle={},    % title
   pdfauthor={Author},     % author     pdfsubject={Subject},   % subject of the document
    pdfcreator={Creator},   % creator of the document
    pdfproducer={Producer}, % producer of the document
   pdfkeywords={keyword1} {key2} {key3}, % list of keywords
   pdfnewwindow=true,      % links in new window
    colorlinks=true,       % false: boxed links; true: colored links
    linkcolor=blue,          % color of internal links
    citecolor=black,        % color of links to bibliography
    filecolor=magenta,      % color of file links
    urlcolor=black    % color of external links
    }
\usepackage{cleveref}
\usepackage[normalem]{ulem}
\usepackage{xcolor}
\usepackage{tikz}
\usepackage{tikz-cd}
\usetikzlibrary{calc, positioning, fit, shapes.misc}
\usepackage{subcaption}
\usepackage{blkarray}
\usepackage{pgfplots}
\pgfplotsset{compat=1.15}
\usepackage{mathrsfs}
\usetikzlibrary{arrows}
\usepackage{comment} %It allows us to use \begin{comment}
\usepackage{float} %It allows us to use [H] in \begin{figure}[H]
\usepackage{multirow} %It allows us to use \multirow in a table

\numberwithin{equation}{section}
\theoremstyle{plain}
   \newtheorem{thm}{Theorem}[section]
   \newtheorem{prop}[thm]{Proposition}
   \newtheorem{lem}[thm]{Lemma}
   \newtheorem{cor}[thm]{Corollary}
   \newtheorem{conj}[thm]{Conjecture}
   \newtheorem{question}[thm]{Question}
   \newtheorem{claim}{Claim}
   \newtheorem*{thm*}{Theorem}
\theoremstyle{definition}
   \newtheorem{defi}[thm]{Definition}
   \newtheorem{exam}[thm]{Example}

\theoremstyle{remark}
   \newtheorem{rmk}[thm]{Remark}

\DeclareMathOperator{\conv}{conv}

\DeclareMathOperator{\peak}{peak}

\DeclareMathOperator{\vrt}{vert}

\newcommand{\oeis}[1]{\href{http://oeis.org/#1}{#1}}
\newcommand{\crown}{\mathcal{C}}

\title{Order polytopes of crown posets}
\author[T. Lundström]{Teemu Lundström}
\author[L.~Saud]{Leonardo Saud Maia Leite}
\address{Department of Mathematics and Systems Analysis, Aalto University, Otakaari 1, 02150, Espoo, Finland}
\email{teemu.lundstrom@aalto.fi}
\address{Department of Mathematics, KTH Royal Institute of Technology, SE-100 44 Stockholm, Sweden}
\email{lsml@kth.se}


\thanks{}

\begin{document}
\definecolor{ududff}{rgb}{0.30196078431372547,0.30196078431372547,1}
\definecolor{ttqqqq}{rgb}{0.2,0,0}

\begin{abstract}
    In the last decade, the order polytope of the zigzag poset has been thoroughly studied. A related poset, called \emph{crown poset}, obtained by adding an extra cover relation between the endpoints of an even zigzag poset, is not so well understood. In this paper, we study the order polytopes of crown posets. We provide explicit formulas for their $f$-polynomials, and show that the sequence of coefficients of each $f$-polynomial is log-concave and asymptotically normal. We provide recursive formulas for their Ehrhart polynomial, giving a counterpart to formulas found in the zigzag case by Petersen--Zhuang (2025). We use these formulas to simplify a computation by Ferroni--Morales--Panova (2025) of the linear term of the order polynomial of these posets. Furthermore, we provide a combinatorial interpretation for the coefficients of the $h^*$-polynomial in terms of the cyclic swap statistic on cyclically alternating permutations, which provides a circular version of a result by Coons--Sullivant (2023).
\end{abstract}

\maketitle
\thispagestyle{empty}
\setcounter{tocdepth}{2}
\tableofcontents
\newpage

\section*{Acknowledgments}
The first author would like to thank Ragnar Freij-Hollanti for helpful discussions. The second author would like to thank Luis Ferroni for the reviews, suggestions, and motivating words, Doriann Albertin, Jai Aslam, and Sampada Kolhatkar, for their contribution in Section~\ref{sec-h-vector}, and Christoph Thäle by observing the connection between the $f$-polynomials studied in this paper and Chebyshev polynomials of the first kind. Both authors thank the anonymous referees for their helpful feedback.

\section{Introduction}
\label{intro}
Order polytopes were introduced by Richard P. Stanley \cite{stanley1986two} as a geometric representation of partially ordered sets (posets). They play a significant role in combinatorial geometry, optimization, and algebraic combinatorics. The \emph{zigzag} (or the \emph{fence}) \emph{poset} $Z_n$ on $[n]$ is defined by the cover relations $1 \prec 2 \succ 3 \prec \cdots \succ n-1 \prec n$ if $n$ is even and $1 \prec 2 \succ 3 \prec \cdots \prec n-1 \succ n$ if $n$ is odd, see Figure \ref{fig:zigzag-poset}, and it has been extensively studied in the literature \cite{coons2019h, zigzag-eulerian}, but the same cannot be said about the crown poset. The \emph{crown poset} $\mathcal{C}_{2n}$ on $[2n]$ is defined by the cover relations $1 \prec 2 \succ 3 \prec \cdots \succ 2n-1 \prec 2n \succ 1$, that is, it corresponds to the zigzag poset $Z_{2n}$ with the extra cover relation $2n \succ 1$, see Figure \ref{fig-crown-poset}.

\begin{figure}[H]
    \centering
    \begin{tikzpicture}[line cap=round,line join=round,>=triangle 45,x=1cm,y=1cm]
        \clip(-16.200778590097016,6.2995524384280617) rectangle (-3,8.529946161913498);
        \draw [line width=1pt] (-14,7)-- (-13,8);
        \draw [line width=1pt] (-13,8)-- (-12,7);
        \draw [line width=1pt] (-9,7)-- (-8,8);
        \draw [line width=1pt] (-8,8)-- (-7,7);
        \draw [line width=1pt] (-7,7)-- (-6,8);
        \draw [line width=1pt] (-6,8)-- (-5,7);
        \draw [line width=1pt] (-12,7)-- (-11,8);
        \draw (-14.074740007757182,7) node[anchor=north west] {1};
        \draw (-13.024426126441885,8.5) node[anchor=north west] {2};
        \draw (-12.020698747926863,7) node[anchor=north west] {3};
        \draw (-11.012736232793635,8.5) node[anchor=north west] {4};
        \draw (-9.066269734520768,7) node[anchor=north west] {1};
        \draw (-8.020190989823679,8.5) node[anchor=north west] {2};
        \draw (-7.02493388454507,7) node[anchor=north west] {3};
        \draw (-6.016971369411841,8.5) node[anchor=north west] {4};
        \draw (-5.0,7) node[anchor=north west] {5};
        \begin{scriptsize}
            \draw [fill=ttqqqq] (-14,7) circle (1.5pt);
            \draw [fill=ttqqqq] (-13,8) circle (1.5pt);
            \draw [fill=ttqqqq] (-12,7) circle (1.5pt);
            \draw [fill=ttqqqq] (-11,8) circle (1.5pt);
            \draw [fill=ttqqqq] (-9,7) circle (1.5pt);
            \draw [fill=ttqqqq] (-8,8) circle (1.5pt);
            \draw [fill=ttqqqq] (-7,7) circle (1.5pt);
            \draw [fill=ttqqqq] (-6,8) circle (1.5pt);
            \draw [fill=ttqqqq] (-5,7) circle (1.5pt);
        \end{scriptsize}
    \end{tikzpicture}
    \caption{The Hasse diagrams of $Z_4$ (left) and $Z_5$ (right).}
    \label{fig:zigzag-poset}
\end{figure}
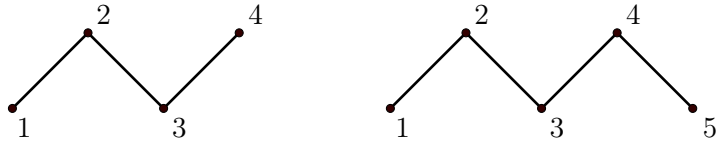

\begin{figure}[H]
    \centering
    \begin{tikzpicture}[line cap=round,line join=round,>=triangle 45,x=1cm,y=1cm]
        \clip(-9,6.2995524384280617) rectangle (-2,8.529946161913498);
        \draw [line width=1pt] (-8,7)-- (-7,8);
        \draw [line width=1pt] (-7,8)-- (-6,7);
        \draw [line width=1pt] (-6,7)-- (-5,8);
        \draw [line width=1pt] (-3,8)-- (-8,7);
        \draw [line width=1pt] (-5,8)-- (-4,7);
        \draw [line width=1pt] (-4,7)-- (-3,8);
        \draw (-8.0,7) node[anchor=north west] {1};
        \draw (-7.0,8.5) node[anchor=north west] {2};
        \draw (-6.0,7) node[anchor=north west] {3};
        \draw (-5.0,8.5) node[anchor=north west] {4};
        \draw (-3.0,8.5) node[anchor=north west] {6};
        \draw (-4.016971369411841,7) node[anchor=north west] {5};
        \begin{scriptsize}
            \draw [fill=ttqqqq] (-8,7) circle (1.5pt);
            \draw [fill=ttqqqq] (-7,8) circle (1.5pt);
            \draw [fill=ttqqqq] (-6,7) circle (1.5pt);
            \draw [fill=ttqqqq] (-5,8) circle (1.5pt);
            \draw [fill=ttqqqq] (-4,7) circle (1.5pt);
            \draw [fill=ttqqqq] (-3,8) circle (1.5pt);
        \end{scriptsize}
    \end{tikzpicture}
    \caption{The Hasse diagram of $\mathcal{C}_6$.}
    \label{fig-crown-poset}
\end{figure}
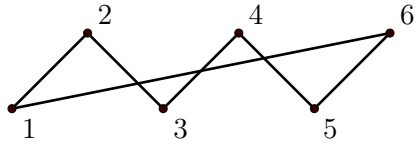
In this paper we study the order polytope of the crown poset $\mathcal{C}_{2n}$. 
The crown poset $\mathcal{C}_{2n}$ is the poset of proper faces of a polygon with $n$ vertices.
The motivation for studying these posets came from studying order polytopes constructed from face lattices of more complicated polytopes.
We will discuss this more in Section \ref{sec-open-further}.
\medskip

The following is an outline of the paper with our main results highlighted.
\medskip

In Section~\ref{section:preliminaries}, we quickly mention the relevant background related to order polytopes and Ehrhart theory, as well as fix some notation.
In Section~\ref{sec-f-vector}, we study the $f$-polynomial of $\mathcal{O}(\mathcal{C}_{2n})$: we provide an explicit formula for each of its coefficients, and use this to prove that these coefficients are a log-concave sequence and also asymptotically normal. These formulas are obtained by using the well-known connection between faces of the order polytope and the so called connected and compatible partitions of the underlying poset. We count these partitions using elementary combinatorial methods.

\begin{thm*}
    (Theorem \ref{therom:formula_for_fvector})
    For all $k \ge 0$,
    \begin{equation*}
        f_k(\mathcal{O}(\crown_{2n})) = \delta_k  + \sum_{i=2}^{2n} \sum_{m = 1}^{\lfloor i/2 \rfloor} \frac{2n}{i} \binom{i}{2m} \binom{n+m-1}{i-1} \binom{2m}{i-k},
    \end{equation*}
    where 
    \begin{equation*}
        \delta_k = 
        \begin{cases}
            2 &\text{ if } k = 0 \\
            1 &\text{ if } k = 1 \\
            0 &\text{ if } k > 1
        \end{cases}\,.
    \end{equation*}
    Hence, for all $n \geq 2$,
    \[
        f(\mathcal{O}(\mathcal{C}_{2n}),t) = (1+t)^2 + t (1+t)^n \sum_{m = 1}^n \frac{n}{m} \binom{n+m-1}{2m-1} (1+t)^m.
    \]
\end{thm*}

In Section~\ref{sec-ehrhart}, we study the Ehrhart polynomial of $\mathcal{O}(\mathcal{C}_{2n})$. We provide a recursive formula for the order polynomial of $\mathcal{C}_{2n}$ in terms of order polynomials of zigzags, which gives a recursive expression for the Ehrhart polynomial of $\mathcal{O}(\mathcal{C}_{2n})$.
This provides a counterpart to the formulas found in the zigzag case in \cite{zigzag-eulerian}. 
This also allows us to simplify the computation for the linear coefficient of the order polynomial $\Omega_{\crown_{2n}}(t)$ made in \cite{ferroni2025skew}.

\begin{thm*}[Theorem \ref{ehrhart-recursion}]
    \begin{align*}
        \Omega_{\mathcal{C}_{2n}}(t) &= \Omega_{\mathcal{C}_{2n}}(t-1) + n \cdot \Omega_{Z_{2n-1}}(t-1) + \Omega_{Z_{2n-3}}(t) + \sum_{1 \leq i < j \leq n} \Omega_{Z_{2(n-j+i)-1}}(t-1) \cdot \Omega_{Z_{2(j-i)-1}}(t) \\
        &= \Omega_{\mathcal{C}_{2n}}(-t) - n \cdot \Omega_{Z_{2n-1}}(-t) + t \cdot \Omega_{Z_{2n-3}}(t) - \sum_{\substack{1 \leq i < j \leq n \\ (i,j) \neq (1,n)}} \Omega_{Z_{2(n-j+i)-1}}(-t) \cdot \Omega_{Z_{2(j-i)-1}}(t)
    \end{align*}
\end{thm*}

In Section~\ref{sec-h-vector}, we study the $h^*$-polynomial of $\mathcal{O}(\mathcal{C}_{2n})$. We give a new combinatorial interpretation of its coefficients in terms of a new permutation statistic that we call \emph{cyclic swap} and discuss some combinatorial properties of this statistic.
This provides a circular version of a result for the zigzag poset studied in \cite{coons2019h}.
\begin{thm*}
    (Theorem \ref{h-characterization})
    For any $n \ge 1$,
   \begin{equation*}
       h^*(\mathcal{O}({\mathcal{C}_{2n}}),t) = \sum_{\sigma \in \operatorname{CA_{2n}}} t^{\operatorname{cswap}(\sigma)}.
    \end{equation*}
\end{thm*}

We also collect together known results to obtain that the $h^*$-polynomial of $\mathcal{O}(\crown_{2n})$ is $\gamma$-nonnegative for every natural number $n$.
\begin{thm*}
    (Theorem \ref{thm:gamma_nonneg})
    For each $n \geq 1$, the polynomial $h^*(\mathcal{O}({\mathcal{C}_{2n}}), t)$ has a nonnegative expansion in the basis $\Gamma_{2n-2}$, that is, there are nonnegative integers $\gamma_{2n,j}$ such that
    \[
    h^*(\mathcal{O}({\mathcal{C}_{2n}}),t) = \sum_{0 \leq j \leq n-1} \gamma_{2n,j} t^j (1+t)^{2n-2-2j}
    \]
    and
    \[
    \gamma_{2n,j} = 2^{-2n+2+2j} |\{ \pi \in \operatorname{JH}(\Tilde{\mathcal{C}}_{2n}, \Tilde{\omega}) \colon \peak(\pi) = j+1 \}|,
    \]
    where $\omega$ is any natural labeling of $\mathcal{C}_{2n}$ and $\peak(\pi) = |\{ i \colon a_{i-1} < a_i > a_{i+1} \}|$.
\end{thm*}

In Section~\ref{sec-open-further}, we present some open questions and further directions.

\section{Preliminaries}\label{section:preliminaries}
In this section we quickly give some conventions and notation used in this paper. We mostly follow standard terminology from polytope theory and Ehrhart theory (see e.g.\ \cite{beck2007computing}).

We consider the empty set $\emptyset$ as a face of a polytope of dimension $-1$. The full polytope is also considered to be a face. The $f$\emph{-vector} of a $d$-polytope $\mathfrak{p}$ is denoted by $f(\mathfrak{p}) = (f_{-1}(\mathfrak{p}),f_0(\mathfrak{p})\dots,f_d(\mathfrak{p}))$, where $f_i(\mathfrak{p})$ is the number of $i$-dimensional faces of $\mathfrak{p}$, and the \emph{$f$-polynomial} $f_\mathfrak{p}(t)$ of $\mathfrak{p}$ is the generating polynomial of $f(\mathfrak{p})$, that is, $f_\mathfrak{p}(t) = \sum_{i=0}^{d+1} f_{i-1}(\mathfrak{p}) t^i$. If $\mathfrak{p} \subseteq \mathbb{R}^d$ is a lattice polytope then the integer point counting function $k \mapsto \# (k\mathfrak{p} \cap \mathbb{Z}^d) \colon \mathbb{Z}_{>0} \to \mathbb{Z}_{\ge 0}$ agrees with a polynomial, called the \emph{Ehrhart polynomial}, which we denote by $L(\mathfrak{p},k)$.
The series $1 + \sum_{k \ge 1} L(\mathfrak{p},k) t^k$ can be written as a rational function
\begin{equation*}
    1 + \sum_{k \ge 1} L(\mathfrak{p},k) t^k = \dfrac{ h_0^{*}(\mathfrak{p}) + h_1^{*}(\mathfrak{p})t+ \cdots + h_d^{*}(\mathfrak{p})t^d}{(1-t)^{d+1}}
\end{equation*}
where the polynomial $h_\mathfrak{p}^{*}(t) = h_0^{*}(\mathfrak{p}) + h_1^{*}(\mathfrak{p})t+ \cdots + h_d^{*}(\mathfrak{p})t^d$ is called the $h^*$\emph{-polynomial} of $\mathfrak{p}$, and the vector $h^{*}(\mathfrak{p}) = (h_0^{*}(\mathfrak{p}), h_1^{*}(\mathfrak{p}), \ldots, h_d^{*}(\mathfrak{p}))$ is the \emph{$h^{*}$-vector} of $\mathfrak{p}$.

Next, we quickly recall basic facts about order polytopes.
For more on order polytopes see \cite{stanley1986two}.
Let $P = ([d],\preceq)$ be a finite poset.
The \emph{order polytope} $\mathcal{O}(P)$ is the polytope
\begin{equation*}
    \{ x = (x_1, \ldots, x_d) \in [0,1]^d \mid x_i \le x_j \text{ for all } i \preceq j \text{ in } P \}.
\end{equation*}
These are full-dimensional polytopes, that is, $\dim(\mathcal{O}(P)) = |P|$.
The vertices of $\mathcal{O}(P)$ are in bijection with the indicator vectors of filters of $P$ and the facet-defining inequalities of $\mathcal{O}(P)$ are given by 
\begin{align*}
    0 \le x_i &\qquad \text{for all $i$ minimal} \\
    x_j \le 1 &\qquad \text{for all $j$ maximal} \\
    x_i \le x_j &\qquad \text{for all cover relations $i \prec j$} \smash{\text{\quad\raisebox{1.2\baselineskip}{.}}}
\end{align*}
Order polytopes also admit unimodular triangulations, which are constructed as follows: let $\sigma$ be a linear extension of $P$, and 
let $v_0^\sigma = \sum_{i=1}^d e_i$. Recursively define $v_i^{\sigma}$ to be $v_i^{\sigma} = v_{i-1}^\sigma - e_{\sigma^{-1}(i)}$ and then define $\Delta^\sigma$ to be the convex hull of $v_0^{\sigma}, v_1^{\sigma}, \ldots, v_d^{\sigma}$. Then the simplices $\Delta^\sigma$ corresponding to all linear extensions $\sigma$ of $P$ form a unimodular triangulation of $\mathcal{O}(P)$ which we will refer to as the \emph{canonical triangulation}.

\section{The \texorpdfstring{$f$}--polynomial of \texorpdfstring{$\mathcal{O}(\crown_{2n})$}-}
\label{sec-f-vector}

We start by studying the $f$-polynomial of $\mathcal{O}(\crown_{2n})$. 
Since $\crown_{2n}$ has $n$ minimal elements, $n$ maximal elements and $2n$ cover relations, it follows that $\mathcal{O}(\crown_{2n})$ has $4n$ facets.
In fact, we can compute explicitly each entry of the $f$-polynomial of $\mathcal{O}(\crown_{2n})$. 
We will approach this by using the well-known characterization of faces of order polytopes in terms of so called connected and compatible partitions of the underlying poset.
We begin by reviewing the related terminology.

Fix a finite poset $P$ and let $\pi = \{ B_1,\dots,B_k \}$ be a partition of the ground set of $P$.
We call the sets $B_i$ either \emph{parts} or \emph{blocks}.
We say that $\pi$ is \emph{connected} if each $B_i$, considered as an induced subposet, is connected.
Recall that a poset is connected if its comparability graph is connected.
We introduce an order relation $\le_\pi$ on the blocks of $\pi$ by setting $B_i \le_\pi B_j$ if there is some $x \in B_i$ and $y \in B_j$ such that $x \le y$ in $P$.
Note that $\le_\pi$ is reflexive but not necessarily transitive or antisymmetric.
If the transitive closure of $\le_\pi$ is antisymmetric, we say $\pi$ is \emph{compatible}.
Note that the transitive closure of $\le_\pi$ is antisymmetric if and only if for all blocks of $\pi$
\begin{equation*}
    B_i \le_\pi  \dots \le_\pi B_j \le_\pi \dots \le_\pi B_i \quad \text{ implies } B_i = B_j.
\end{equation*}
A sequence of blocks of the form $B_i \le_\pi  \dots \le_\pi B_j \le_\pi \dots \le_\pi B_i$ is called a \emph{cycle}.
Hence, $\pi$ is compatible if and only if there are no cycles among the blocks of $\pi$, other than cycles consisting of a single block.
We often write just $\le$ instead of $\le_\pi$.
If $\pi$ is a connected and compatible partition we call $\pi$ a \emph{CCP} for short.
If $|\pi| = k$ and $\pi$ is a CCP we call $\pi$ a $k$-CCP.

The faces of $\mathcal{O}(P)$ turn out to correspond to, not CCPs of $P$, but CCPs of $\hat P = P \cup \{ \hat 0, \hat 1 \}$, which is the poset obtained from $P$ by adding a minimum element $\hat 0$ and a maximum element $\hat 1$.

\begin{thm}[\cite{stanley1986two}]
    Let $P$ be a finite poset of size $d$.
    Let $0 \le k \le d$.
    The $k$-faces of $\mathcal{O}(P)$ are in bijection with the $(k+2)$-CCPs of $\hat P$.
\end{thm}

We turn to studying CCPs of $\crown_{2n}$ and $\hat \crown_{2n}$.
The CCPs of $\crown_{2n}$ are easy to recognize:

\begin{lem}\label{lemma:compatible_iff_at_least_one_odd_part}
    A connected partition of $\crown_{2n}$ with at least two blocks is compatible if and only if at least one block has odd cardinality.
\end{lem}

\begin{proof}
    First, suppose $\pi$ is a connected partition of $\crown_{2n}$.
    Write the elements of $\crown_{2n}$ as
    \begin{equation*}
        a_1 < b_1 > a_2 < b_2 > \cdots < b_{n} > a_1. \tag{$\ast$}
    \end{equation*}
    Since each block of $\pi$ is connected, each block is going to consist of (cyclically) consecutive elements in the list $(\ast)$.
    If $B \le B'$ for some $B,B' \in \pi$ and $B \neq B'$ then there needs to be elements $a \in B$ and $b \in B'$ that are (cyclically) consecutive in the list $(\ast)$.
    Hence the only way to have $B \le B'$ for two distinct blocks is if the union $B \cup B'$ consists of (cyclically) consecutive elements.
    Thus, if there are any cycles among the blocks of $\pi$, any element in $\crown_{2n}$ is going to appear in one of those blocks.
    Therefore, any cycle among the blocks of $\pi$ is going to contain all blocks.
    
    With this observation we now fix a connected partition $\pi = \{ B_1,\dots,B_k \}$ of $\crown_{2n}$ where $k \ge 2$ and proceed to prove the lemma.
    
    Suppose $\pi$ has an odd block $B_i$.
    Either $B_i$ is of the form
    \begin{equation*}
        B_i = \{ a_{i_1} < b_{i_2} > \cdots < b_{i_m} > a_{i_{m+1}} \}
    \end{equation*}
    or $B_i$ is of the form
    \begin{equation*}
        B_i = \{ b_{i_1} > a_{i_2} < \cdots > a_{i_m} < b_{i_{m+1}} \}.
    \end{equation*}
    In the first case there cannot be any block $B_j \neq B_i$ with $B_j \le B_i$.
    In the second case there cannot be any block $B_j \neq B_i$ with $B_i \le B_j$.
    Hence $B_i$ cannot appear in any non-trivial cycle.
    We conclude that $\pi$ is compatible.
    
    Then suppose every block of $\pi$ has even cardinality.
    We aim to show that $\pi$ is not compatible.
    Now either each $B_i$ is of the form
    \begin{equation*}
        B_i = \{ a_{i_1} < b_{i_2} > \cdots > a_{i_{m}} < b_{i_{m+1}}  \}
    \end{equation*}
    or each $B_i$ is of the form
    \begin{equation*}
        B_i = \{ b_{i_1}  > a_{i_2} < \cdots < b_{i_m} > a_{i_{m+1}} \}.
    \end{equation*}
    Hence we can chain together all the blocks of $\pi$ to form a cycle $B_1 \le B_2 \le \dots \le B_k \le B_1$.
    Since $k \ge 2$, we conclude that $\pi$ is not compatible.
\end{proof}

The comparability graph of $\crown_{2n}$ is the cycle graph on $2n$ elements.
By Lemma~\ref{lemma:compatible_iff_at_least_one_odd_part}, the CCPs of $\crown_{2n}$ therefore correspond to connected partitions of the vertex set of the cycle graph where at least one part has odd cardinality.
Counting such partitions seems combinatorially feasible.
However, the faces of $\mathcal{O}(\crown_{2n})$ do not correspond to CCPs of $\crown_{2n}$ but instead to the CCPs of $\hat \crown_{2n} = \crown_{2n} \cup \{ \hat 0 , \hat 1 \}$. The next proposition shows how to obtain the CCPs of $\hat \crown_{2n}$ from the CCPs of $\crown_{2n}$.

\begin{prop}\label{prop:k-CCP's_and_pairs_(P,S)}
    For all $k \ge 2$ let
    \begin{equation*}
        \mathcal{A}_k = \{ (P,S) \colon P \text{ is a  CCP of } \crown_{2n}, \ S \subseteq \{ \text{odd parts of  } P  \}, \ |P| - |S| = k -2  \}.
    \end{equation*}
    \begin{enumerate}[(i)]
        \item For all $k \geq 2$ there is an injection $\mathcal{A}_k \longrightarrow \{ \text{$k$-CCP's of } \hat \crown_{2n} \}$.
        \item For all $k \geq 3$ the injection is also surjective and thus 
        \begin{equation*}
            \# \{ \text{$k$-CCP's of } \hat \crown_{2n} \} = |\mathcal{A}_k|.
        \end{equation*}
        \item For $k = 2$ the injection misses two elements and thus 
        \begin{equation*}
            \# \{ \text{2-CCP's of } \hat \crown_{2n} \} = |\mathcal{A}_2| + 2.    
        \end{equation*}
    \end{enumerate}
\end{prop}

\begin{figure}[t]
    \centering
    \includegraphics[scale=0.5]{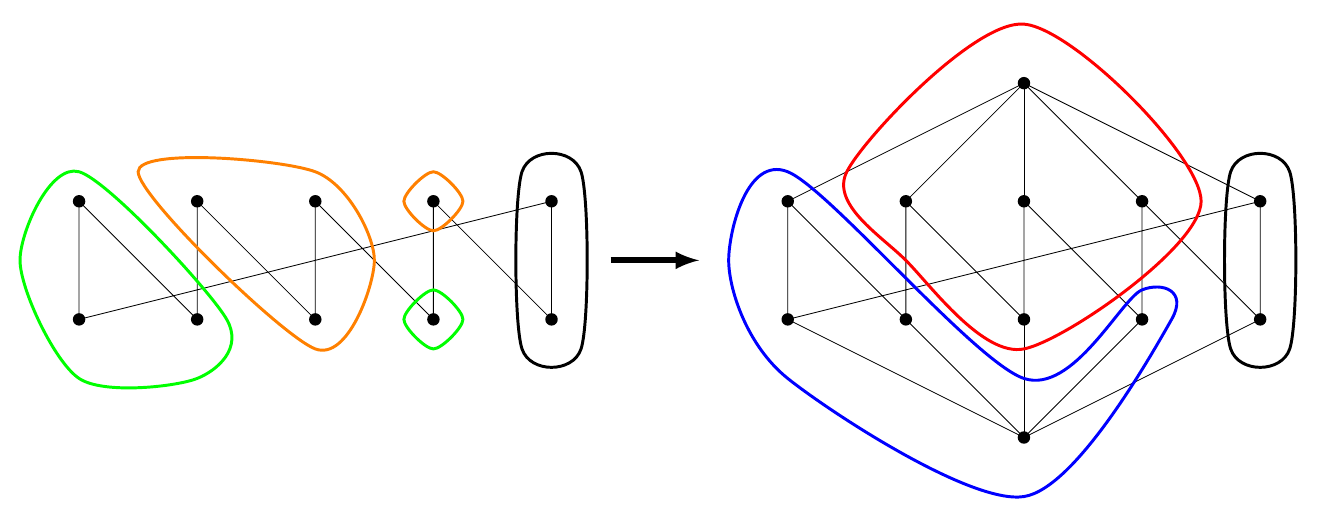}
    \caption{
        The map from the proof of Proposition~\ref{prop:k-CCP's_and_pairs_(P,S)}.
        On the left we have a CCP of $\crown_{2n}$ which then maps to a CCP of $\hat \crown_{2n}$ on the right.
        Using the notation from the proof, on the left the green parts are $B_1,\dots,B_m$, the orange parts are $D_1,\dots,D_\ell$, $S$ is the set of all odd parts in $P$ and the black part is the only part in $P \setminus S$.
        On the right the blue part is $B_1 \cup \dots \cup B_m \cup \{ \hat 0 \}$ and the red part is $D_1 \cup \dots \cup D_{\ell} \cup \{ \hat 1 \}$.
        Here $n = 5$, $k = 3$, $m = 2$, $\ell = 2$.
        Note that here $S$ is the set of \emph{all} odd parts of $P$ but this is not always the case.}
    \label{fig:partition_map}
\end{figure}

\begin{proof}
    Let $k \ge 2$.
    Given a pair $(P,S) \in \mathcal{A}_k$ we let $S_0 = \{ B_1, \dots, B_m \} \subseteq S$ be the set of parts in $S$ that have more lower elements than upper elements and $S_1 = \{ D_1,\dots,D_{\ell} \} \subseteq S$ be the set of parts that have more upper elements than lower elements.
    Here $S = S_0 \sqcup S_1$.
    We map $(P,S)$ to 
    \begin{equation*}
        \pi = (P \setminus S) \cup \{ B_1 \cup \dots \cup B_m \cup \{ \hat 0 \} \} \cup \{ D_1 \cup \dots \cup D_\ell \cup \{ \hat 1 \} \},
    \end{equation*}
    see Figure~\ref{fig:partition_map}.
    Our aim is to show that this mapping is a well-defined injection $\mathcal{A}_k \to \{ \text{$k$-CCP's of } \hat \crown_{2n}  \}.$
    First we show that this map is well-defined.
    Use the same notation as above.
    Our aim is to show that $\pi$ is a $k$-CCP of $\hat \crown_{2n}$.
    Since $P$ is a partition of $\crown_{2n}$, it is easy to see that $\pi$ is a partition of $\hat \crown_{2n}$.
    The parts in $P \setminus S$ are connected since $P$ was connected, and the new parts $B_1 \cup \dots \cup B_m \cup \{ \hat 0 \}$ and $D_1 \cup \dots \cup D_\ell \cup \{ \hat 1 \}$ are clearly connected.
    Here $|\pi| = |P| - |S| + 2 = k$.
    It remains to show that $\pi$ is compatible.
    Since each $B_i$ is connected and has more lower elements than upper elements, each $B_i$ looks something like in Figure~\ref{fig:blocks_a}.
    Similarly, since each $D_i$ is connected and has more upper elements than lower elements, each $D_i$ looks something like in Figure~\ref{fig:blocks_b}.
    In particular, no part of $P$ can be below any $B_i$, and no part of $P$ can be above any $D_i$.
    Write $P \setminus S = \{ C_1, \dots, C_r \}$.
    Furthermore, write $B = B_1 \cup \dots \cup B_m \cup \{ \hat 0 \}$ and $D = D_1 \cup \dots \cup D_{\ell} \cup \{ \hat 1 \}$.
    Thus $\pi = \{ C_1, \dots, C_r,B,D \}$.
    Since $P$ is a compatible partition of $\crown_{2n}$, there cannot be any cycles among $C_1, \dots, C_r$.
    Hence any cycle among the parts of $\pi$ is going to either contain $B$ or $D$, or both.
    If we have $C_i \le B$ for any $i$ then $C_i \le B_j$ for some $j$, which is not possible as we saw.
    Similarly, if $D \le C_i$ for any $i$ then $D_j \le C_i$ for some $j$, which is also not possible.
    We also cannot have $D \le B$.
    Therefore the only possible cycles among the parts of $\pi$ have the forms
    \begin{equation*}
        B \le C_\bullet \le \dots \le C_\bullet  \qquad
        C_\bullet \le \dots \le C_\bullet \le D  \qquad
        B \le C_\bullet \le \dots \le C_\bullet \le D,
    \end{equation*}
    where $\bullet$ ranges over some (possibly empty) subset of the indices $\{ 1,\dots,r \}$.
    But none of these are cycles.
    We conclude that $\pi$ is compatible and the map is well-defined.

    \begin{figure}
    \begin{subfigure}{.5\textwidth}
        \centering
        \includegraphics[width=1\linewidth]{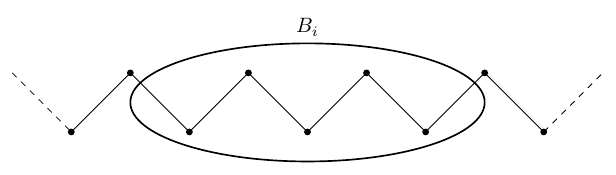}
        \caption{What $B_i$ looks like.}
        \label{fig:blocks_a}
    \end{subfigure}
    \begin{subfigure}{.5\textwidth}
        \centering
        \includegraphics[width=1\linewidth]{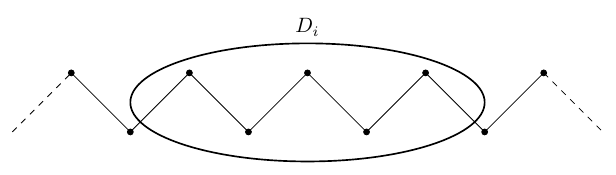}
        \caption{What $D_i$ looks like.}
        \label{fig:blocks_b}
    \end{subfigure}
    \caption{Blocks from the proof of Proposition~\ref{prop:k-CCP's_and_pairs_(P,S)}.}
    \label{fig:blocks}
\end{figure}
    
    Next, we show injectivity.
    Let $(P,S) \in \mathcal{A}_k$ be as above and use the same notation as before.
    Let $(P',S') \in \mathcal{A}_k$ be another pair and use similar notation.
    The pair $(P,S)$ gives rise to a partition $\pi = \{ C_1, \dots, C_r,B,D \}$ and the pair $(P',S')$ gives rise to a partition $\pi' = \{ C'_1,\dots,C_{r'}', B',D' \}$.
    We assume $\pi = \pi'$ and aim to show $(P,S) = (P',S')$.
    For all $i=1,\dots,m$  we have $C_i \not\in \{ B',D' \}$ since $\hat 0 \in B'$ and $\hat 1 \in D'$.
    Therefore $C_i \in \{ C'_1,\dots,C'_{r'} \}$.
    Thus $\{ C_1, \dots, C_r \} \subseteq \{ C'_1,\dots,C'_{r'} \}$.
    Similarly, $\{ C'_1,\dots,C'_{r'} \} \subseteq \{ C_1, \dots, C_r \}$.
    Therefore
    \begin{equation*}
        P \setminus S = \{ C_1, \dots, C_r \} = \{ C'_1,\dots,C'_{r'} \} = P' \setminus S'.
    \end{equation*}
    We have $B \in \pi'$ but we cannot have $B \in \{ C'_1, \dots, C'_{r'},D'  \}$ as $\hat 0 \in B$.
    Thus $B = B'$, that is,
    \begin{equation*}
        B_1 \cup \dots \cup B_m \cup \{ \hat 0 \} = B'_1 \cup \dots \cup B'_{m'} \cup \{ \hat 0 \}
    \end{equation*}
    and therefore $B_1 \cup \dots \cup B_m = B'_1 \cup \dots \cup B'_{m'}$.
    Consider any $B_i \in \{ B_1, \dots, B_m \}$.
    Now $B_i \subseteq B'_1 \cup \dots \cup B'_{m'}$.
    \begin{claim}
        The union of any two or more parts $B'_1,\dots,B'_{m'}$ is disconnected.
    \end{claim}
    \begin{proof}[Proof of claim]
        In any connected subset of $\crown_{2n}$ the difference between the number of upper elements and lower elements is either 0 or 1.
        Each $B'_j$ has one more lower element than upper elements.
        Hence in any union of two or more parts $B'_1,\dots,B'_{m'}$ the difference between the number of lower elements and upper elements is at least 2 and hence such a union cannot be connected.
    \end{proof}
    Since $B_i$ is connected and is contained in the union $B'_1 \cup \dots \cup B'_{m'}$ we need to have $B_i \subseteq B'_j$ for some $j \in [m']$.
    By symmetry, $B'_j$ is contained in some  $B_1,\dots,B_m$, say $B'_j \subseteq B_t$.
    Now $B_i \subseteq B'_j \subseteq B_t$.
    Since $B_1, \dots, B_m$ are pairwise disjoint, we need to have $B_i = B_t$ and hence $B_i = B'_j$.
    We have now seen that $\{ B_1, \dots, B_m \} \subseteq \{ B'_1,\dots,B'_{m'} \}$.
    By symmetry, $\{ B'_1,\dots,B'_{m'} \} \subseteq \{ B_1, \dots, B_m \}$.
    Therefore $\{ B_1, \dots, B_m \} = \{ B'_1,\dots,B'_{m'} \}$.
    A similar argument shows that $\{ D_1,\dots,D_{\ell} \} = \{ D'_1,\dots,D'_{\ell'} \}$.
    Therefore
    \begin{equation*}
        S = \{ B_1, \dots, B_m \} \cup \{ D_1,\dots,D_{\ell} \} = \{ B'_1,\dots,B'_{m'} \} \cup \{ D'_1,\dots,D'_{\ell'} \} = S'.
    \end{equation*}
    Combining this with $P \setminus S = P' \setminus S'$ we obtain $(P,S) = (P',S')$ as desired.
    We have now shown that the map is injective, proving part (i).
    
    We now turn to parts (ii) and (iii).
    We still let $k \ge 2$.
    Let
    \begin{equation*}
        \pi = \{ E_1,\dots,E_{k-2},E_{\hat 0},E_{\hat 1} \}
    \end{equation*}
    be a $k$-CCP of $\hat \crown_{2n}$ where $E_{\hat 0}$ is the part containing $\hat 0$ and $E_{\hat 1}$ is the part containing $\hat 1$.
    Note that $\hat 0$ and $\hat 1$ cannot be in the same part: otherwise, if $\hat 0, \hat 1 \in E_i$ and $E_j$ is any other part then we would have $E_i \le E_j \le E_i$ which is not possible since $\pi$ is compatible.
    Furthermore, we assume that $\pi$ is neither of the partitions shown in Figure~\ref{fig:forbidden_partitions}.
     \begin{figure}
        \begin{subfigure}{.5\textwidth}
            \centering
            \includegraphics[width=0.7\linewidth]{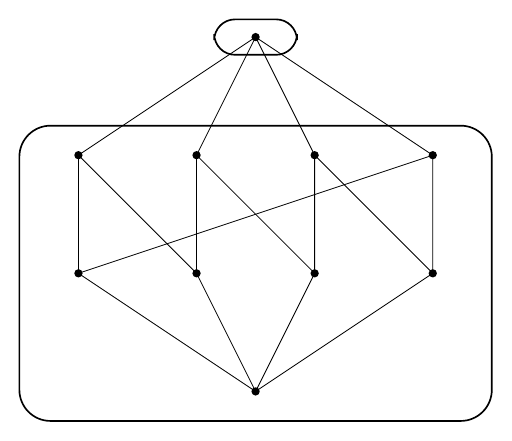}
            \caption{}
            \label{fig:forbidden_a}
        \end{subfigure}
        \begin{subfigure}{.5\textwidth}
            \centering
            \includegraphics[width=0.7\linewidth]{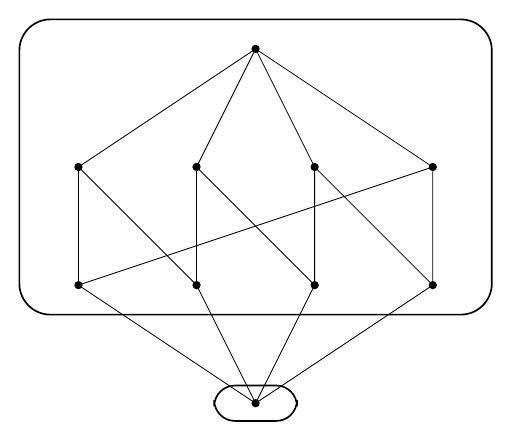}
            \caption{}
            \label{fig:forbidden_b}
        \end{subfigure}
        \caption{The two partitions we assume $\pi$ is not equal to in the proof of Proposition~\ref{prop:k-CCP's_and_pairs_(P,S)}. Here $n =4$.}
        \label{fig:forbidden_partitions}
    \end{figure}
    We think of $E_{\hat 0} \setminus \{ \hat 0 \}$ as an induced subgraph of $\crown_{2n}$ and we let $K_1, \dots, K_s$ be its connected components.
    Similarly, we think of $E_{\hat 1} \setminus \{ \hat 1 \}$ as an induced subgraph of $\crown_{2n}$ and we let $H_1, \dots, H_p$ be its connected components, see Figure~\ref{fig:parition_map_inverse}.
    We let
    \begin{align*}
        P &= \{ E_1,\dots,E_{k-2},K_1,\dots,K_s,H_1,\dots,H_p \} \smash{\text{\quad\raisebox{-0.5\baselineskip}{.}}}\\
        S &= \{ K_1, \dots, K_s,H_1, \dots, H_p \}
    \end{align*}
    Here $S \subseteq P$ and $|P| - |S| = k-2$.
    To show that the pair $(P,S)$ is in the domain of our mapping, we need to show that
    \begin{enumerate}[(1)]
        \item $P$ is a CCP of $\crown_{2n}$, and
        \item every part in $S$ has odd cardinality.
    \end{enumerate}
    It is clear that $P$ is a connected partition of $\crown_{2n}$.
    The next claim helps us show that $|K_i|$ is odd for all $i=1,\dots,s$.
    \begin{claim}
        For all upper elements $x$ and all $i=1,\dots,s$, if $x \in K_i$ then $K_i$ contains both lower elements below $x$.
    \end{claim}
    \begin{proof}[Proof of claim]
        Suppose towards contradiction that $y$ is a lower element with $y < x$ and $y \not\in K_i$.
        Here $y$ cannot be in any of the $K_1, \dots, K_s$ since otherwise it would be in the same connected component as $x$ and hence in $K_i$.
        Therefore $y \not\in E_{\hat 0}$.
        Let $E_i$ be the part of $\pi$ that contains $y$.
        Here $E_i \neq E_{\hat 0}$.
        But now we have $E_i \le E_{\hat 0} \le E_i$, which contradicts the fact that $\pi$ is compatible.
        We conclude that $y \in K_i$.
    \end{proof}
    Hence whenever any $K_i$ contains and upper element, it contains both of the lower elements below it.
    We cannot have $K_i = \crown_{2n}$ for any $i$ as otherwise $E_{\hat 0} = \crown_{2n} \cup \{ \hat 0 \}$ and $E_{\hat 1} = \{ \hat 1 \}$, which would contradict our assumption about $\pi$.
    Thus each $K_i$ is going to have more lower elements than upper elements, and since each $K_i$ is connected, $|K_i|$ has to be odd for all $i$.
    A similar claim shows that whenever any $H_i$ contains any lower element $y$ then $H_i$ needs to contain both upper elements above $y$.
    Similarly as before, we cannot have $H_i = \crown_{2n}$ for any $i$ as otherwise $E_{\hat 1} = \crown_{2n} \cup \{ \hat 1 \}$ and $E_{\hat 0} = \{ \hat 0 \}$ contrary to our assumption about $\pi$.
    Thus each $H_i$ is going to have more upper elements than lower elements, and since each $H_i$ is connected,  $|H_i|$ has to be odd for all $i$.
    We have now shown (2).
    To see that we have (1) it remains to see that $P$ is a compatible partition of $\crown_{2n}$.
    If $S \neq \emptyset$ then $P$ has an odd part, and since $P$ has at least two parts, by Lemma~\ref{lemma:compatible_iff_at_least_one_odd_part} the partition $P$ is a compatible partition of $\crown_{2n}$.
    If $S = \emptyset$ then $P = \{ E_1,\dots,E_{k-2} \}$ and $P$ is a CCP of $\crown_{2n}$, and in particular $P$ is compatible.
    We have now shown (1).
    Therefore $(P,S) \in \mathcal{A}_k$.
    Recall that $K_1, \dots, K_s$ have more lower elements than upper elements and $H_1, \dots, H_p$ have more upper elements than lower elements.
    It is now easy to see that $(P,S)$ maps to $\pi$.
    We have now seen that when $k \ge 3$ the map is surjective, proving part (ii).
    
    \begin{figure}[t]
        \centering
        \includegraphics[scale=0.6]{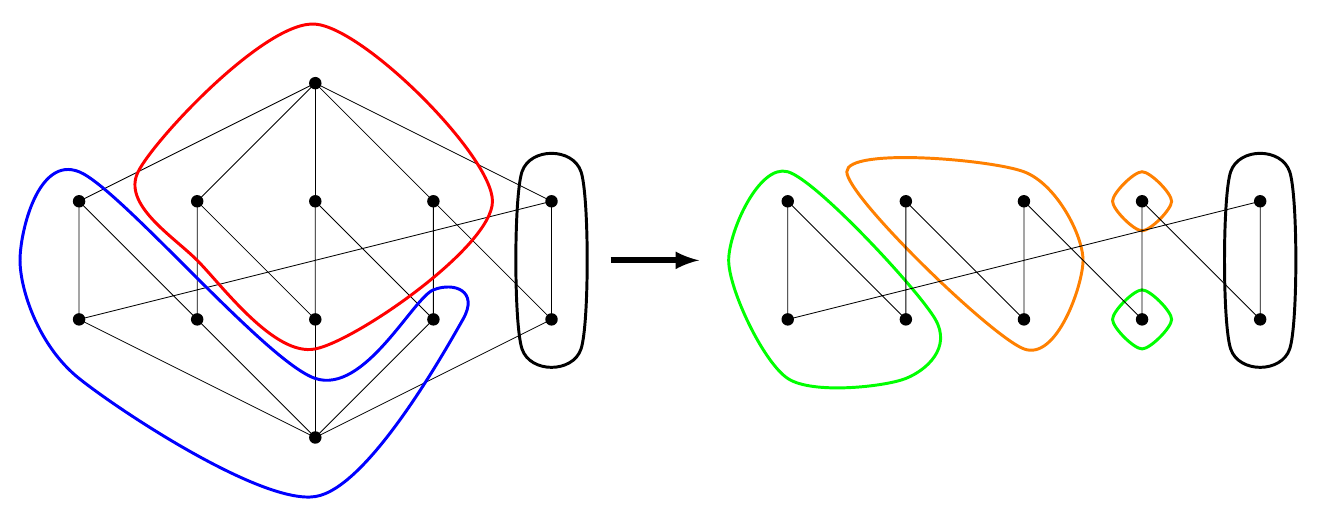}
        \caption{
            The inverse of the map from the proof of Proposition~\ref{prop:k-CCP's_and_pairs_(P,S)}.
            On the left we have a CCP of $\hat \crown_{2n}$ and on the right a corresponding CCP of $\crown_{2n}$.
            Using the notation from the proof, the blue part is $E_{\hat 0}$, the red part is $E_{\hat 1}$, the green parts are $K_1,\dots,K_s$, the orange parts are $H_1, \dots, H_p$ and the black parts are $E_1,\dots,E_{k-2}$.
            Here $n = 5$, $k = 3$, $s = 2$, $p = 2$.
         }
        \label{fig:parition_map_inverse}
    \end{figure}

    Finally, let $k = 2$.
    We claim that no pair $(P,S)$ can map to either of the partitions shown in Figure~\ref{fig:forbidden_partitions}.
    Suppose towards contradiction that $(P,S) \in \mathcal{A}_2$ maps to $\{ \crown_{2n} \cup \{ \hat 1 \}, \{ \hat 0 \} \}$.
    Thus
    \begin{equation*}
        (P \setminus S) \cup \{ B_1 \cup \dots \cup B_m \cup \{ \hat 0 \} \} \cup \{ D_1 \cup \dots \cup D_\ell \cup \{ \hat 1 \} \} = \{ \crown_{2n} \cup \{ \hat 1 \}, \{ \hat 0 \}\}
    \end{equation*}
    where $B_1, \dots, B_m$ are the parts in $S$ that have more lower elements than upper elements and $D_1,\dots,D_\ell$ are the parts in $S$ that have more upper elements than lower elements.
    But this implies $D_1 \cup \dots \cup D_m = \crown_{2n}$ and hence $\crown_{2n}$ would have more upper elements than lower elements, which is not the case.
    Hence no pair $(P,S)$ maps to the partition $\{ \crown_{2n} \cup \{ \hat 1 \}, \{ \hat 0 \} \}$.
    A similar argument shows that no pair $(P,S)$ maps to the partition $\{ \crown_{2n} \cup \{ \hat 0 \} , \{ \hat 1 \} \}$.
    Therefore, when $k = 2$ the map is an injective map 
    \begin{equation*}
        \mathcal{A}_2 \longrightarrow \{ 2\text{-CCPs of } \hat  \crown_{2n} \} \setminus \{ \text{the two partitions shown in Figure~\ref{fig:forbidden_partitions}} \}
    \end{equation*}
    and we saw that this map is also surjetive.
    We have now shown part (iii) finishing the proof.
\end{proof}

Next, we look at how to compute the cardinalities of the sets $\mathcal{A}_k$ for all $k \ge 2$.
Recall from Lemma~\ref{lemma:compatible_iff_at_least_one_odd_part} that if $P$ is a connected partition of $\crown_{2n}$ with at least two parts, then it is compatible if and only if it has at least one odd part.
By splitting into cases depending on the cardinality of $P$ and the number of odd parts $P$ has, we can write the set 
\begin{equation*}
    \mathcal{A}_k = \{ (P,S) \colon P \text{ is a  CCP of } \crown_{2n}, \ S \subseteq \{ \text{odd parts of  } P  \}, \ |P| - |S| = k -2  \}
\end{equation*}
as a disjoint union
\begin{equation*}
    \bigsqcup_{(i,j)} 
    \left\{ 
    (P,S) \ \vline
    \begin{array}{l}
        P \text{ is a connected partition of } \crown_{2n}, \  |P| = i, \ P \text{ has $j$ odd parts,} \\
        S \subseteq \{ \text{odd parts of } P \}, \ |S| = i-k+2
    \end{array}
    \right\} \tag{$\ast\ast$}
\end{equation*}
where
\begin{equation*}
    (i,j) \in \{ (1,0), (2,1), (2,2), (3,1), (3,2), (3,3), \dots \}.
\end{equation*}
We observe that if we have a connected partition of $\crown_{2n}$ with $i$ parts of which $j$  are odd, then 
\begin{align*}
    2n &= (2k_1 + 1) + \dots + (2k_j + 1) + 2m_1 + \dots + 2m_{i-j} \\
    &= 2(k_1 + \dots + k_j) + j + 2(m_1 + \dots + m_{i-j})
\end{align*} 
and hence $j$ is even.
A partition of $\crown_{2n}$ with $i$ parts of which $j$ are odd satisfies $j \le i \le 2n$.
Therefore in the disjoint union $(\ast\ast)$ the pairs $(i,j)$ can be set to range over the set
\begin{equation*}
    \{ (1,0), (2,2), (3,2), (4,2), (4,4), (5,2), (5,4),(6,2),(6,4),(6,6), \dots, (2n,2n) \},
\end{equation*}
which we can write as 
\begin{equation*}
    I_n \coloneqq \{ (1,0) \} \cup \{ (i,2m) \in \mathbb{Z} \times \mathbb{Z} \mid 2 \le i \le 2n , \ 1 \le m \le i/2\}.
\end{equation*}
We now turn to counting the cardinality of the disjoint union $(\ast\ast)$.
For any pair $(i,j) \in I_n$ we can first choose a connected partition $P$ of $\crown_{2n}$ with $i$ parts of which $j$ are odd, and then choose the subset $S$.
No matter what the choice of $P$ is, the number of ways of choosing $S$ is always the same, namely, $\binom{j}{i-k+2}$.
If we let $p_n(i,j)$ be the number of connected partitions of $\crown_{2n}$ with $i$ parts of which $j$ are odd, then the disjoint union $(\ast\ast)$, and hence $\mathcal{A}_k$, has cardinality
\begin{equation*}
    \sum_{(i,j) \in I_n} p_n(i,j) \binom{j}{i-k+2} = p_n(1,0)\binom{0}{1-k+2} + \sum_{i=2}^{2n} \sum_{m=1}^{\lfloor i/2 \rfloor} p_n(i,2m) \binom{2m}{i-k+2}.
\end{equation*}
Here we use the convention that $\binom{a}{b} = 0$ if $b < 0$. 
Let us record this as a separate result for future reference.
Note that $p_n(1,0) = 1$.

\begin{lem}\label{lemma:card_of_Ak}
    For all $k \ge 2$, 
    \begin{equation*}
        |\mathcal{A}_k| = \binom{0}{3-k} + \sum_{i=2}^{2n} \sum_{m=1}^{\lfloor i/2 \rfloor} p_n(i,2m) \binom{2m}{i-k+2}.
    \end{equation*}
\end{lem}

Next, we compute the numbers $p_n(i,2m)$.
The comparability graph of $\crown_{2n}$ is the cycle graph with $2n$ nodes.
Denote this graph by $O_{2n}$.
The connected partitions of $\crown_{2n}$ are thus the same as the connected partitions of $O_{2n}$, that is, partitions of the vertex set of $O_{2n}$ such that each part induces a connected subgraph.
Hence, in order to compute the numbers $p_n(i,2m)$ we need to count how many connected partitions of $O_{2n}$ there are that have $i$ parts of which $2m$ are odd.

\begin{lem}\label{lemma:p_n(i,j)}
    For all $2 \le i \le 2n$ and $1 \le m \le i/2$,
    \begin{equation*}
        p_n(i,2m) = \frac{2n}{i} \binom{i}{2m} \binom{n+m-1}{i-1}.
    \end{equation*}
\end{lem}

\begin{proof}
    Partitioning $O_{2n}$ into $i \ge 2$ connected parts is the same as removing $i$ edges.
    Hence the connected partitions of $O_{2n}$ into $i$ parts of which $2m$ are odd are in bijection with a certain set
    \begin{equation*}
        M \subseteq \binom{E(O_{2n})}{i}
    \end{equation*}
    of $i$-sets of edges of $O_{2n}$.
    Here the sets in $M$ correspond to those sets of  edges that are removed from $O_{2n}$ when creating the corresponding partition.
    We need to count $|M|$.
    
    As $|S| = i$ for all $S \in M$, the number of pairs $(S,e)$ with $S \in M$ and $e \in S$ is $|M| \cdot i$.
    Hence
    \begin{equation*}
        |M| = \frac{1}{i} \# \{ (S,e) \mid S \in M, \ e \in S \} = \frac{1}{i}  \sum_{e \in E(O_{2n})} \# \{ S \in M \mid e \in S \}.
    \end{equation*}
    Removing any single edge $e \in E(O_{2n})$ from $O_{2n}$ leaves us with the path graph $P_{2n}$ on $2n$ vertices.
    Therefore for any $e \in E(O_{2n})$
    \begin{equation*}
        \# \{ S \in M \mid e \in S \} = \# \{ \text{connected partitions of } P_{2n} \text{ with $i$ parts of which $2m$ are odd} \}
    \end{equation*}
    and hence
    \begin{equation*}
        |M| = \frac{2n}{i} \#\{ \text{connected partitions of } P_{2n} \text{ with $i$ parts of which $2m$ are odd} \}.
    \end{equation*}
    We turn to counting the cardinality of the last set above.
    
    Given a connected partition of $P_{2n}$, the parts are essentially subpaths of the graph.
    By adding a vertex to each odd subpath, we obtain a map
    \begin{equation*}
        \left\{ 
        \begin{array}{c}
            \text{connected  partitions of $P_{2n}$} \\
            \text{into $i$ parts} \\
            \text{of which $2m$ are odd}
        \end{array}
        \right\}
        \longrightarrow
        \left\{ 
        \begin{array}{c}
            \text{connected partitions of $P_{2n+2m}$} \\
            \text{into $i$ even parts}
        \end{array}
        \right\},
    \end{equation*}
    see Figure~\ref{figure:adding_vertices}.
    \begin{figure}[t]
        \centering
        \includegraphics[scale=0.6]{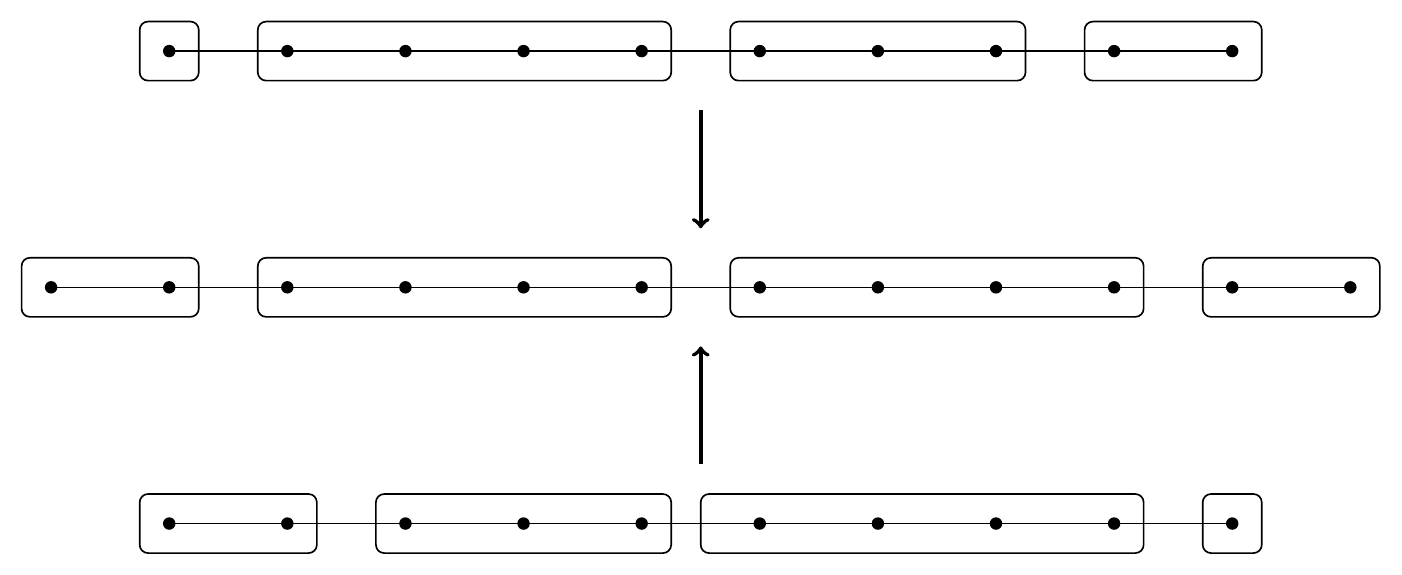}
        \caption{Adding a vertex to each odd subpath of $P_{2n}$ to obtain partitions of $P_{2n+2m}$ into $i$ even parts.
            Here $2n = 10, i = 4$ and  $2m = 2$.
            The top and bottom pictures are partitions of $P_{2n}$ and the middle picture is a partition of $P_{2n+2m}$.}
        \label{figure:adding_vertices}
    \end{figure}
    Given a partition form the set on the right-hand side, by selecting $2m$ subpaths from which we remove one vertex, we see that the fibers of this map all have size $\binom{i}{2m}$.
    Hence
    \begin{equation*}
        \# \left\{ 
        \begin{array}{c}
            \text{connected  partitions of $P_{2n}$} \\
            \text{into $i$ parts} \\
            \text{of which $2m$ are odd}
        \end{array}
        \right\}
        =
        \binom{i}{2m} 
        \# \left\{ 
        \begin{array}{c}
            \text{connected partitions of $P_{2n+2m}$} \\
            \text{into $i$ even parts}
        \end{array}
        \right\}.
    \end{equation*}
    We turn to counting the cardinality of the last set above.
    We obtain this from the following general result.
    
    \begin{claim}
        There are $\binom{\ell-1}{j-1}$ connected partitions of $P_{2\ell}$ into $j$ even parts.
    \end{claim}
    \begin{proof}[Proof of claim.]
        Color every other edge blue and every other edge red, see Figure~\ref{figure:colored_path_graph}.
        There are $\ell$ blue edges and $\ell-1$ red edges.
        The connected partitions of $P_{2\ell}$ into $j$ even parts now correspond to removing any $j-1$ red edges.
    \end{proof}
    By this claim we obtain that
    \begin{equation*}
        \# \left\{ 
        \begin{array}{c}
            \text{connected partitions of $P_{2n+2m}$} \\
            \text{into $i$ even parts}
        \end{array}
        \right\} 
        = 
        \binom{n+m-1}{i-1}.
    \end{equation*}
    Putting everything together, we have obtained that the number of connected partitions of $O_{2n}$ into $i$ parts of which $2m$ are odd, equals
    
    \begin{equation*}
        \frac{2n}{i} \binom{i}{2m} \binom{n+m-1}{i-1}.
    \end{equation*}
\end{proof}

\begin{figure}[t]
    \centering
    \includegraphics[scale=1]{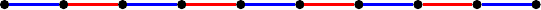}
    \caption{Colored path graph $P_{10}$}
    \label{figure:colored_path_graph}
\end{figure}
    
Using all of the results obtained in this section we are now able to find a formula for the $f$-polynomials of $\mathcal{O}(\crown_{2n})$. 

\begin{thm}\label{therom:formula_for_fvector}
    For all $k \ge 0$,
    \begin{align}
        f_k(\mathcal{O}(\crown_{2n})) &= \delta_k  + \sum_{i=2}^{2n} \sum_{m = 1}^{\lfloor i/2 \rfloor} \frac{2n}{i} \binom{i}{2m} \binom{n+m-1}{i-1} \binom{2m}{i-k} \label{eq:F1}\\
        & = \delta_k  + \sum_{m=1}^{n} \sum_{i = 2m}^{2n} \frac{2n}{i} \binom{i}{2m} \binom{n+m-1}{i-1} \binom{2m}{i-k} \label{eq:F2} \smash{\text{\quad\raisebox{1.2\baselineskip}{.}}}
    \end{align}
    where 
    \begin{equation*}
        \delta_k = 
        \begin{cases}
            2 &\text{ if } k = 0 \\
            1 &\text{ if } k = 1 \\
            0 &\text{ if } k > 1
        \end{cases}\,.
    \end{equation*}
    Hence, for all $n \geq 2$,
    \begin{equation}
    \label{eq:F3}
        f(\mathcal{O}(\mathcal{C}_{2n}),t) = (1+t)^2 + t (1+t)^n \sum_{m = 1}^n \frac{n}{m} \binom{n+m-1}{2m-1} (1+t)^m.
    \end{equation}
\end{thm}

\begin{proof}
    We start by proving \eqref{eq:F1}. First, let $k = 0$. We know that $f_0(\mathcal{O}(\crown_{2n}))$ equals the number of $2$-CCP's of $\hat \crown_{2n}$, which by Proposition~\ref{prop:k-CCP's_and_pairs_(P,S)}(iii) equals $|\mathcal{A}_2| + 2$.
    By Lemma~\ref{lemma:card_of_Ak} and Lemma~\ref{lemma:p_n(i,j)} we obtain
    \begin{align*}
        2+ |\mathcal{A}_2| &= 2 + \binom{0}{1} + \sum_{i=2}^{2n} \sum_{m=1}^{\lfloor i/2 \rfloor} p_n(i,2m)\binom{2m}{i} \\
        &= 2  +  \sum_{i=2}^{2n} \sum_{m = 1}^{\lfloor i/2 \rfloor} \frac{2n}{i} \binom{i}{2m} \binom{n+m-1}{i-1} \binom{2m}{i} \smash{\text{\quad\raisebox{1.2\baselineskip}{.}}}
    \end{align*}
    
    Then let $k \ge 1$.
    The number  $f_k(\mathcal{O}(\crown_{2n}))$ equals the number of $k+2$ CCP's of $\hat \crown_{2n}$, which by Proposition~\ref{prop:k-CCP's_and_pairs_(P,S)} equals $|\mathcal{A}_{k+2}|$.
    By Lemma~\ref{lemma:card_of_Ak} and Lemma~\ref{lemma:p_n(i,j)} we obtain
    \begin{align*}
        |\mathcal{A}_{k+2}|  &= \binom{0}{1-k} + \sum_{i=2}^{2n} \sum_{m = 1}^{\lfloor i/2 \rfloor} p_n(i,2m) \binom{2m}{i-k} \\
        &= \binom{0}{1-k} + \sum_{i=2}^{2n} \sum_{m = 1}^{\lfloor i/2 \rfloor} \frac{2n}{i} \binom{i}{2m} \binom{n+m-1}{i-1} \binom{2m}{i-k} \smash{\text{\quad\raisebox{1.2\baselineskip}{.}}}
    \end{align*}
    Here $\binom{0}{1-k} = 1$ when $k=1$ and $\binom{0}{1-k} = 0$ when $k > 1$, finishing the proof of \eqref{eq:F1}.

    Since
    \begin{equation*}
        \{ (i,m) \in \mathbb{Z}^2 \mid 2 \le i \le 2n, \ 1 \le m \le i /2  \} = \{ (i,m) \in \mathbb{Z}^2 \mid 1 \le m \le n, \ 2m \le i \le 2n \}
    \end{equation*}
    we can also write 
    \begin{equation*}
        f_k(\mathcal{O}(\crown_{2n})) = \delta_k  + \sum_{m=1}^{n} \sum_{i = 2m}^{2n} \frac{2n}{i} \binom{i}{2m} \binom{n+m-1}{i-1} \binom{2m}{i-k},
    \end{equation*}
    which verifies \eqref{eq:F2}.
    
    Now, let us prove \eqref{eq:F3}. By definition,
    \[
        f(\mathcal{O}(\mathcal{C}_{2n}),t) = 1 + 2t + t^2 + t \cdot \sum_{k=2}^{2n} \sum_{m=1}^{n} \sum_{i=2m}^{2n}  \frac{2n}{i} \binom{i}{2m} \binom{n+m-1}{i-1} \binom{2m}{i-k} t^k.
    \]
    Since
    \[
        \sum_{k=2}^{2n} \binom{2m}{i-k} t^k = \sum_{j=0}^{2m} \binom{2m}{j} t^{i-j} = t^{i-2m} (1+t)^{2m}
    \]
    and
    \[
        \frac{2n}{i} \binom{i}{2m} \binom{n+m-1}{i-1} = \frac{n}{m} \binom{n+m-1}{2m-1} \binom{n-m}{i-2m},
    \]
    it follows that
    \[
        f(\mathcal{O}(\mathcal{C}_{2n}),t) = 1 + 2t + t^2 + t \cdot \sum_{m=1}^{n} \frac{n}{m} \binom{n+m-1}{2m-1} (1+t)^{2m} \left[ \sum_{i=2m}^{2n} \binom{n-m}{i-2m} t^{i-2m} \right].
    \]
    Since
    \begin{align*}
        & \sum_{m = 1}^n \frac{n}{m} \binom{n+m-1}{2m-1} (1+t)^{2m} \left[ \sum_{i=2m}^{2n} \binom{n-m}{i-2m} t^{i-2m} \right] = \smash{\text{\qquad\raisebox{-3.5\baselineskip}{,}}} \\
        & \sum_{m = 1}^n \frac{n}{m} \binom{n+m-1}{2m-1} (1+t)^{2m} \left[ \sum_{j=0}^{n+m} \binom{n-m}{j} t^{j} \right] = \\
        & \sum_{m = 1}^n \frac{n}{m} \binom{n+m-1}{2m-1} (1+t)^{2m} (1+t)^{n-m} = \\
        & (1+t)^n \sum_{m = 1}^n \frac{n}{m} \binom{n+m-1}{2m-1} (1+t)^m
    \end{align*}
    we conclude
    \[
        f(\mathcal{O}(\mathcal{C}_{2n}),t) = (1+t)^2 + t (1+t)^n \sum_{m = 1}^n \frac{n}{m} \binom{n+m-1}{2m-1} (1+t)^m.
    \]
\end{proof}

\begin{table}[H]
\centering
\begin{tabular}{|| c || c ||} 
 \hline
 $n$ & $f(\mathcal{O}({\mathcal{C}_{2n}}))$ \\
 \hline
 \hline
 2 & (1, 7, 17, 18, 8, 1) \\
 \hline
 3 & (1, 18, 73, 129, 116, 54, 12, 1) \\ 
 \hline
 4 & (1, 47, 265, 656, 896, 730, 360, 104, 16, 1) \\
 \hline
 5 & (1, 123, 881, 2810, 5170, 6045, 4672, 2405, 810, 170, 20, 1) \\
 \hline
 6 & (1, 322, 2785, 10884, 25228, 38517, 40692, 30408, 16140, 6018, 1532, 252, 24, 1) \\
 \hline
\end{tabular}
\caption{Table of $f$-vectors of $\mathcal{O}(\mathcal{C}_{2n})$ for $2 \leq n \leq 6$.}
\label{table-f-vector}
\end{table}

We finish this section by finding another expression for the number of vertices of $\mathcal{O}(\crown_{2n})$. Recall that the \emph{Lucas sequence} is an integer sequence defined recursively by $L(0) = 2$, $L(1) = 1$, and $L(n) = L(n-1) + L(n-2)$ if $n > 1$. It is closely related to the cycle graph $O_n$ since it corresponds to the number of matchings  and to the number of independent vertex sets and vertex covers of $O_n$  for $n \ge 3$ \cite[\oeis{A000032}]{oeis}.
From this we deduce:

\begin{cor}
    The number of vertices of $\mathcal{O}(\crown_{2n})$ is $L(2n)$. In particular, for every $n \geq 0$, 
    \[
        L(2n) = \left( \frac{1+\sqrt{5}}{2} \right)^{2n} + \left( \frac{1-\sqrt{5}}{2} \right)^{2n} = 2 + \sum_{i=2}^{2n} \sum_{m = 1}^{\lfloor i/2 \rfloor} \frac{2n}{i} \binom{i}{2m} \binom{n+m-1}{i-1} \binom{2m}{i} \quad .
    \]
\end{cor}
\begin{proof}
    The first equality is well-known and follows from the definition of Lucas sequence. For the second equality, observe that the number of vertices in an order polytope of a poset $P$ is the number of filters in $P$, which equals the number of antichains in $P$.
    Hence, it suffices to compute the number of antichains in $\crown_{2n}$. The Hasse diagram of $\crown_{2n}$ as a graph is also its own comparability graph. Since the number of independent sets in a comparability graph of a poset is the number of antichains in that poset, there are $L(2n)$ antichains in $\crown_{2n}$.
    The formula now follows from Theorem~\ref{therom:formula_for_fvector} using $k=0$.
\end{proof}

\subsection{The log-concavity of the \texorpdfstring{$f$}--vector}

Let $\mathcal{S} = \{a_k\}_{k=0}^n$ be a sequence of real numbers. We say that 
\begin{itemize}
    \item $\mathcal{S}$ is \emph{unimodal} if there exists $0 \leq m \leq n$ such that $a_0 \leq \cdots \leq a_{m-1} \leq a_m \geq a_{m+1} \geq \cdots \geq a_n$;
    \item $\mathcal{S}$ is \emph{log-concave} if $a_k^2 \geq a_{k+1} a_{k-1}$ for all $k=1, \ldots, n-1$;
    \item $\mathcal{S}$ is \emph{ultra log-concave} if $\left\{ \nicefrac{a_k}{\binom{n}{k}} \right\}_{k=0}^n$ is log-concave;
    \item the generating polynomial of $\mathcal{S}$, $p_\mathcal{S}(t) = a_0 + a_1 t + \ldots a_n t^n$, is \emph{real-rooted} if $p_\mathcal{S}(t)$ has only real-roots.
\end{itemize}

The following classical lemma connects the four properties above:

\begin{lem}
\label{lem:sequences}
    Let $\mathcal{S} = \{a_k\}_{k=0}^n$ be a sequence of nonnegative real numbers. If
    \begin{enumerate}[(i)]
        \item $p_\mathcal{S}(t)$ is real-rooted, then $\mathcal{S}$ is ultra log-concave;
        \item $\mathcal{S}$ is ultra log-concave, then $\mathcal{S}$ is log-concave;
        \item $\mathcal{S}$ is log-concave and $a_k > 0$ for all $k = 0, 1, \ldots, n$, then $\mathcal{S}$ is unimodal.
    \end{enumerate}
\end{lem}

In this section, we prove the log-concavity of the entries of the $f$-vector of $\mathcal{O}(\mathcal{C}_{2n})$. The following simple lemma will be useful in the proof of Theorem \ref{conj-f-vector}:

\begin{lem}
\label{lem:sequences2}
    Let $b_0, b_1 \geq 0$. If $\mathcal{S} = \{a_k\}_{k=0}^n$ is a ultra log-concave sequence of nonnegative real numbers such that $a_0 \geq b_0$, $a_1 \geq 2b_1$, then $\{a_0 + b_0, a_1 + b_1, a_2, \ldots, a_n\}$ is a log-concave sequence.
\end{lem}

\begin{proof}
    By Lemma \ref{lem:sequences}(ii), it is sufficient to show that
    \begin{equation}
    \label{eq:LC1}
        (a_1 + b_1)^2 \geq (a_0 + b_0) \cdot a_2
    \end{equation}
    and
    \begin{equation}
    \label{eq:LC2}
        a_2^2 \geq (a_1 + b_1) \cdot a_3.
    \end{equation}
    
    In fact, suppose that \eqref{eq:LC1} does not hold, that is, $(a_0 + b_0) \cdot a_2 > (a_1 + b_1)^2$. By ultra log-concavity of $\mathcal{S}$, it follows that
    \[
        \frac{a_1^2}{\binom{n}{1}^2} \geq \frac{a_0}{\binom{n}{0}} \cdot \frac{a_2}{\binom{n}{2}} \implies a_1^2 \geq \frac{2n}{n-1} a_0 a_2 \geq 2 a_0 a_2.
    \]
    Hence,
    \begin{gather*}
        (a_0 + b_0) \cdot a_2 > (a_1 + b_1)^2 \geq 2 a_0 a_2 + 2 a_1 b_1 + b_1 ^2 \implies \smash{\text{\quad\raisebox{-0.7\baselineskip}{.}}}\\
        (b_0 - a_0) \cdot a_2 > 2 a_1 b_1 + b_1^2
    \end{gather*}
    By hypothesis, $a_0 \geq b_0$, from which follows that $0 \geq a_2 > 2 a_1 b_1 + b_1^2$, which is a contradiction because $a_1, b_1 \geq 0$. Therefore, \eqref{eq:LC1} holds.
    
    Now, suppose that \eqref{eq:LC2} does not hold, that is, $(a_1 + b_1) \cdot a_3 > a_2^2$. Again, by ultra log-concavity of $\mathcal{S}$,
    \[
        \frac{a_2^2}{\binom{n}{2}^2} \geq \frac{a_1}{\binom{n}{1}} \cdot \frac{a_3}{\binom{n}{3}} \implies a_2^2 \geq \frac{3(n-1)}{2(n-2)} a_1 a_3 \geq \frac{3}{2} a_1 a_3 .
    \]
    Hence,
    \begin{gather*}
        (a_1 + b_1) \cdot a_3 > a_2^2 \geq \frac{3}{2} a_1 a_3 \implies \smash{\text{\quad\raisebox{-0.7\baselineskip}{,}}}\\
        0 > a_3 \left( \frac{a_1}{2} - b_1 \right)
    \end{gather*}
    which is a contradiction, since $a_1 \geq 2b_1$ implies $a_3 \left( \frac{a_1}{2} - b_1 \right) \geq 0$. Therefore, \eqref{eq:LC2} holds.
\end{proof}

The next result will not be used throughout this paper, but it can be seen as a generalization of Lemma \ref{lem:sequences2}:

\begin{prop}
\label{prop:sequences3}
    Let $\mathcal{S} = \{a_k\}_{k=0}^n$ and $\mathcal{T} = \{b_k\}_{k=0}^n$ be sequences of nonnegative real numbers such that $\mathcal{S}$ is ultra log-concave and $\mathcal{T}$ is log-concave. If $a_k \geq 2(k+1) \cdot b_k$ for all $k \in \{0, 1, \ldots, n\}$, then $\mathcal{U} = \{a_k + b_k\}_{k=0}^n$ is a log-concave sequence.
\end{prop}

\begin{proof}
    We will show that
    \begin{equation}
    \label{eq:LC3}
        (a_k + b_k)^2 \geq (a_{k-1} + b_{k-1}) \cdot (a_{k+1} + b_{k+1})
    \end{equation}
    for all $k \in [n-1]$.

    In fact, suppose that \eqref{eq:LC3} does not hold for some $i \in [n-1]$, that is, $(a_{i-1} + b_{i-1}) \cdot (a_{i+1} + b_{i+1}) > (a_i + b_i)^2$. By ultra log-concavity of $\mathcal{S}$,
    \[
        \frac{a_i^2}{\binom{n}{i}^2} \geq \frac{a_{i-1}}{\binom{n}{i-1}} \cdot \frac{a_{i+1}}{\binom{n}{i+1}} \implies  a_i^2 \geq \frac{i+1}{i} \frac{n-i+1}{n-i} a_{i-1} a_{i+1} \geq \frac{i+1}{i} a_{i-1} a_{i+1}.
    \]
    Hence,
    \begin{gather}
        (a_{i-1} + b_{i-1}) \cdot (a_{i+1} + b_{i+1}) > (a_i + b_i)^2 \geq \frac{i+1}{i} a_{i-1} a_{i+1} + 2 a_i b_i + b_i^2 \implies \smash{\text{\quad\raisebox{-0.7\baselineskip}{.}}} \nonumber\\
        a_{i+1} \left( b_{i-1} - \frac{a_{i-1}}{2i} \right) + a_{i-1} \left( b_{i+1} - \frac{a_{i+1}}{2i} \right) > 2 a_i b_i + (b_i^2 - b_{i-1} b_{i+1}) \label{eq:LC4}
    \end{gather}
    By hypothesis, $\frac{a_{i-1}}{2i} \geq b_{i-1}$ and $a_{i+1} \geq 2(i+2) \cdot b_{i+1} \geq 2i \cdot b_{i+1}$, which implies that $\frac{a_{i+1}}{2i} \geq b_{i+1}$. Thus, the left hand side of \eqref{eq:LC4} is smaller or equal than zero, which implies that
    \[
        0 > 2 a_i b_i + (b_i^2 - b_{i-1} b_{i+1}).
    \]
    However, this is a contradiction since $a_i, b_i \geq 0$ and $\mathcal{T}$ is log-concave. Therefore, the proposition follows.
\end{proof}

\begin{rmk}
    Neither Lemma \ref{lem:sequences2} nor Proposition \ref{prop:sequences3} holds true if we try to prove that the new sequence obtained is ultra log-concave. In fact, take $\mathcal{S}_1 = \{ 1/10, 2, 2, 1, 1/10, 1/250 \}$ and $\mathcal{T}_1 = \{ 0, 1, 0, 0, 0, 0 \}$. Then $\mathcal{S}_1$ is ultra log-concave and $\mathcal{T}_1$ is log-concave, but the sequence $\mathcal{U}_1 = \{ 1/10, 3, 2, 1, 1/10, 1/250 \}$ is not ultra log-concave since
    \[
        \left( \frac{2}{\binom{5}{2}} \right)^2 < \frac{3}{\binom{5}{1}} \cdot \frac{1}{\binom{5}{3}}.
    \]
    Another counterexample is $\mathcal{S}_2 = \{ 1, 4, 3, 1, 1/10, 1/250 \}$ and $\mathcal{T}_2 = \{ 1, 1, 0, 0, 0, 0 \}$. Again, $\mathcal{S}_2$ is ultra log-concave and $\mathcal{T}_2$ is log-concave, but $\mathcal{U}_2 = \{ 2, 5, 3, 1, 1/10, 1/250 \}$ is not ultra log-concave since
    \[
        \left( \frac{3}{\binom{5}{2}} \right)^2 < \frac{5}{\binom{5}{1}} \cdot \frac{1}{\binom{5}{3}}.
    \]
\end{rmk}

Table \ref{table-f-vector} presents the $f$-vectors of $\mathcal{O}(\mathcal{C}_{2n})$ for small values of $n$. Using a computer, it is easy to verify that $f(\mathcal{O}(\crown_{2n}))$ is log-concave for large values of $n$. The next result, which was independently proved by Hai-Ling Li, Shi-Mei Ma, and Ze-Chao Qu in \cite{liface}, shows that this is not just a coincidence:

\begin{thm}
\label{conj-f-vector}
    For all $n \geq 2$, the entries of the $f$-vector of $\mathcal{O}(\mathcal{C}_{2n})$ form a log-concave sequence.
\end{thm}

We need an additional result to prove Corollary \ref{conj-f-vector}. The \emph{Chebyshev polynomials of first kind} $T_n(t)$ \cite{mason2002chebyshev} are defined by the recurrence relation
\begin{equation*}
    T_n(t) \coloneqq 
    \begin{cases}
        1 &\mbox{ if } n = 0 \\
        t &\mbox{ if } n = 1 \\
        2t \cdot T_{n-1}(t) - T_{n-2}(t) &\mbox{ if } n > 1 
    \end{cases}
\end{equation*}
and its closed formula is
\begin{equation}
\label{eq:cheb1}
    T_n(t) = n \sum_{m=0}^n (-2)^m \frac{(n+m-1)!}{(n-m)! \cdot (2m)!} (1-t)^m = \frac{(t + \sqrt{t^2 -1})^n + (t - \sqrt{t^2 -1})^n}{2} \quad .
\end{equation}

The \emph{Chebyshev polynomials of second kind} $U_n(t)$ are defined by the recurrence relation
\begin{equation*}
    U_n(t) \coloneqq 
    \begin{cases}
        1 &\mbox{ if } n = 0 \\
        2t &\mbox{ if } n = 1 \\
        2t \cdot T_{n-1}(t) - T_{n-2}(t) &\mbox{ if } n > 1 
    \end{cases}
\end{equation*}
and its closed formula is
\begin{equation}
\label{eq:cheb2}
    U_n(t) = \sum_{m=0}^n (-2)^m \frac{(n+m+1)!}{(n-m)! \cdot (2m + 1)!} (1-t)^m = \frac{(t + \sqrt{t^2 -1})^{n+1} - (t - \sqrt{t^2 -1})^{n+1}}{2\sqrt{t^2 -1}} \quad .
\end{equation}

We now list some properties of Chebyshev polynomials. The Chebyshev polynomials of first and second kinds are related by the following expression:
\begin{equation}
\label{eq:cheb3}
    T_n '(t) = n \cdot U_{n-1}(t) \quad \mbox{for every } n \geq 1.
\end{equation}

The Chebyshev polynomials of first and second kinds are real-rooted \cite[Chapter 2.2]{mason2002chebyshev}. Moreover, on the interval $[-1,1]$, all of the extrema of $T_n(t)$ have values that are either $-1$ or $1$ for all $n \geq 1$. More specifically, let $d > 2$ and $k > 0$ be such that $d \mid 2n$, $0 < k < d/2$ and $\gcd (k,d) = 1$. Then,
\begin{itemize}
    \item for even $n$, $T_n(t) = 1$ if and only if $t = \pm 1$ or $t = \cos \left( \frac{2k \pi}{d} \right)$, with $\frac{2n}{d}$ even, and $T_n(t) = -1$ if and only if $t = \cos \left( \frac{2k \pi}{d} \right)$, with $\frac{2n}{d}$ odd, from which follows that there are $n/2 + 1$ such values of $t$;
    \item for odd $n$, $T_n(t) = 1$ if and only if $t = 1$ or $t = \cos \left( \frac{2k \pi}{d} \right)$, with $\frac{2n}{d}$ even, and $T_n(t) = -1$ if and only if $t = \cos \left( \frac{2k \pi}{d} \right)$, with $\frac{2n}{d}$ odd, from which follows that there are $(n+1)/2$ such values of $t$.
\end{itemize}
Other interesting properties of $T_n(t)$ and $U_n(t)$ can be found in the literature, and we refer to \cite{mason2002chebyshev} for further reading.

\begin{lem}
\label{lem:f-vector}
    For all $n \geq 2$,
    \[
        f(\mathcal{O}(\mathcal{C}_{2n}),t) = (1+t)^2 + 2 t \cdot (1+t)^n \left( T_n \left(1 + \frac{1+t}{2} \right) - 1 \right) .
    \]
\end{lem}

\begin{proof}
    Observe that
    \begin{equation}
    \label{eq:F4}
        2 \cdot T_n \left( 1 + \frac{y}{2} \right) = 2n \sum_{m=0}^n (-2)^m \frac{(n+m-1)!}{(n-m)! \cdot (2m)!} \left( - \frac{y}{2} \right)^m = 2 + \sum_{m=1}^n \frac{n}{m} \binom{n+m-1}{2m - 1} y^m.
    \end{equation}
    By comparing Equations \ref{eq:F3} and \ref{eq:F4}, and setting $y = t+1$, we get
    \[
        f(\mathcal{O}(\mathcal{C}_{2n}),t) = (1+t)^2 + t \cdot (1+t)^n \left( 2 \cdot T_n \left(1 + \frac{1+t}{2} \right) - 2 \right) ,
    \]
    which proves the lemma.
\end{proof}

\begin{proof}[Proof of Theorem \ref{conj-f-vector}]
    Fix $n \geq 2$, and write
    \[
        f(\mathcal{O}(\mathcal{C}_{2n}),t) = \sum_{k=0}^{2n+1} f_{k-1} t^k
    \]
    and
    \[
        2 t \cdot (1+t)^n \left( T_n \left(1 + \frac{1+t}{2} \right) - 1 \right) = \sum_{k=1}^{2n+1} a_{k-1} t^k.
    \]
    Hence,
    \[
        (f_{-1}, f_0, f_1, f_2, \ldots, f_{2n}) = (1, a_0 + 2, a_1 + 1, a_2, \ldots, a_{2n})
    \]
    by Lemma \ref{lem:f-vector}. We first show that $Q_n(t) \coloneqq 2t \cdot (1+t)^n \left( T_n \left(1 + \frac{1+t}{2} \right) - 1 \right)$ is a real-rooted polynomial and then we use Lemma \ref{lem:sequences2} to conclude the proof.
    
    In fact, suppose that $n$ is even. Then the equation $T_n(t) = 1$ has only real solutions, since, as observed previously, $T_n(t) = 1$ has $n/2 + 1$ distinct solutions, of which two are $1$ and $-1$, and the others are solutions in the open interval $(-1,1)$. Hence, $n/2 - 1$ among the distinct roots of $P_n(t) \coloneqq T_n(t) - 1$ are local maxima, from which follows that the multiplicity of each of them is at least $2$. Since the degree of $P_n(t)$ is $n$ and $1$ and $-1$ are also roots of $P_n(t)$, it follows that $P_n(t)$ is a real-rooted polynomial. Similarly, we can show that $P_n(t)$ is real-rooted if $n$ is odd. Therefore $P_n \left(1 + \frac{1+t}{2} \right)$ is also real-rooted, which implies that $Q_n(t)$ is a real-rooted polynomial.

    Hence, by Lemma \ref{lem:sequences}(i), $(a_0, a_1, a_2, \ldots, a_{2n})$ is a ultra log-concave sequence and, by Lemma \ref{lem:sequences2}, by taking $b_0 = 2$ and $b_1 = 1$, it follows that $(a_0 + 2, a_1 + 1, a_2, \ldots, a_{2n})$ is a log-concave sequence. In fact, $a_0, a_1 \geq 2$ follows from Theorem \ref{therom:formula_for_fvector}. Moreover, $f_0^2 \geq f_1$ because, if $V$ and $E$ are the vertex set and the edge set of $\mathcal{O}(\mathcal{C}_{2n})$, respectively, then $E \subseteq \binom{V}{2}$, from which follows that $f_1 = |E| \leq |V|^2 = f_0^2$.
\end{proof}

\subsection{A probabilistic result: the Central Limit Theorem for the \texorpdfstring{$f$}--vector}
Let $a_0, \ldots, a_n \geq 0$ be such that $\sum_{k=0}^n (a_k) = 1$, and let $X$ be a discrete random variable taking values in $\{0, 1, \ldots, n\}$ such that $\mathbb{P}(X = k) = a_k$ for all  $k = 0, 1, \ldots, n$. The \emph{probability generating polynomial} of $X$ is defined as
\[
    p_X(t) = a_0 + a_1 t + \ldots + a_n t^n.
\]
Observe that the \emph{mean} of $X$ is given by
\[
    \mu_X = \mathbb{E}(X) = \sum_{k=0}^n k \cdot \mathbb{P}(X = k) = p_X ' (1)
\]
and its \emph{variance} is
\[
    \sigma^2_X = \mathrm{Var}(X) = \mathbb{E}(X^2) - [\mathbb{E}(X)]^2 = p_X '' (1) + p_X ' (1) - [p_X ' (1)]^2.
\]

The following result connects the theory of real-rooted polynomials and Probability Theory, and has been used to prove asymptotic normality of many sequences in Combinatorics, see e.g. \cite{bender1973, moussa2012, mirko2009, harper1967stirling, china2023}.

\begin{thm}\cite[Theorem 2]{bender1973}
\label{th:bender}
    Let $\{X_n\}_{n \geq 0}$ be a sequence of random variables taking values in $\{0, 1, \ldots, n\}$ such that
    \begin{enumerate}[(i)]
        \item $p_{X_n}(t) = \sum_{k = 0}^n a_{n,k} t^k$ is real-rooted for all $n \geq 0$, and
        \item $\sigma^2_{X_n} \to \infty$.
    \end{enumerate}
    Then
    \[
        \frac{X_n - \mu_{X_n}}{\sigma_{X_n}} \xrightarrow{D} N(0,1) \quad \mbox{when} \quad n \to \infty .
    \]
\end{thm}

If a sequence of discrete random variables $\{X_n\}_{n \geq 0}$ converges in distribution to $N(0,1)$, we say that the numbers $\{a_{n,k}\}_{n, k \geq 0}$ are \emph{asymptotically normal}. Equivalently, following Bender \cite[Equation $(1.1)$]{bender1973}, $\{a_{n,k}\}_{n, k \geq 0}$ (or $\{X_n\}_{n \geq 0}$) are asymptotically normal if and only if
\begin{equation}
\label{eq:asym-normal}
    \lim_{n \to \infty} \sup_{t \in \mathbb{R}} \left\{ \left| \sum_{k \leq \sigma_{X_n} \cdot t + \mu_{X_n}} (a_{n,k}) - \frac{1}{2 \sqrt{\pi}} \int_{- \infty}^t e^{-\frac{x^2}{2}} dx \right| \right\} = 0 \quad .
\end{equation}
In this section, we will show that the coefficients of $\{ f(\mathcal{O}(\mathcal{C}_{2n}),t) \}_{n \geq 2}$ are asymptotically normal.

Let $Y_1, \ldots, Y_n$ be random variables and let $w_1, \ldots, w_n \geq 0$ be real numbers such that $\sum_{k=1}^n w_k = 1$. The \emph{mixture distribution of $Y_1, \ldots, Y_n$ with weights $w_1, \ldots, w_k$} is the random variable $X$ whose probability generating polynomial is given by
\[
    p_X(t) = \sum_{k=1}^n w_k p_{Y_k}(t).
\]
Its expectation and its variance are given, respectively, by
\begin{equation}
\label{eq:mix}
    \mu_X = \sum_{k=1}^n w_k \mu_{Y_k} \quad \mbox{and} \quad \sigma^2_X = \sum_{k=1}^n w_k \left[ \sigma^2_{Y_k} + (\mu_{Y_k} - \mu_X)^2 \right].
\end{equation}

We use mixture distributions and Theorem \ref{th:bender} to show that the coefficients of $T_n(1 + \frac{1+t}{2}) - 1$ are asymptotically normal:

\begin{cor}
\label{cor:normal-cheb}
    Let $\{i_n\}_{n \geq 1}$ and $\{j_n\}_{n \geq 1}$ be sequences of nonnegative integers, and let
    \[
        p_n(t) = t^{i_n} (1+t)^{j_n} \sum_{m=1}^n \frac{n}{m} \binom{n+m-1}{2m - 1} (1+t)^m
    \]
    for every $n \geq 1$. Then the sequence of coefficients of $\{p_n(t)\}_{n \geq 1}$ are asymptotically normal.
\end{cor}

\begin{proof}
    We prove the claim by showing that the polynomials $p_n(t)$ satisfy the hypothesis of Theorem \ref{th:bender} for every $n \geq 1$. In fact, by the proof of Lemma \ref{lem:f-vector}, $p_n(t)$ is real-rooted for all $n \geq 1$. We now find a lower bound for $\sigma^2_{X_n}$, the variance of the random variable $X_n$ associated to $p_n(t) / p_n(1)$.

    In fact, fix $n \in \mathbb{N}$, and set
    \[
        \theta_{n,m} \coloneqq \frac{n}{m} \binom{n+m-1}{2m - 1} 2^{j_n+m}, \quad \theta_n \coloneqq \sum_{m=1}^n \theta_{n,m} \quad \mbox{and} \quad w_{n,m} \coloneqq \frac{\theta_{n,m}}{\theta_n}.
    \]
    Observe that the generating probability polynomial of a binomial random variable $Y_m \sim i_n + \mathrm{Bin} \left( j_n + m, \frac{1}{2} \right)$ is given by
    \[
        p_{Y_m}(t) = t^{i_n} \left( \frac{1+t}{2} \right)^{j_n+m}.
    \]
    Recall that
    \[
        \sigma^2_{Y_m} = \frac{j_n + m}{4} \geq \frac{m}{4}.
    \]
    Since
    \[
        \frac{p_n(t)}{p_n(1)} = \sum_{m=1}^n w_{n,m} \cdot p_{Y_m}(t),
    \]
    for each $n \geq 1$, it follows that $X_n$ has the probability distribution of the mixture distribution of $Y_1, \ldots, Y_n$ with weights $w_{n,1}, \ldots, w_{n,n}$. Hence, by \eqref{eq:mix},
    \[
        \sigma^2_{X_n} \geq \sum_{m=1}^n w_{n,m} \sigma^2_{Y_m} \geq \frac{1}{4} \sum_{m=1}^n m \cdot w_{n,m}.
    \]
    To finish the proof, we give a lower bound to $\sum_{m=1}^n m \cdot w_{n,m}$. In fact,
    \begin{align*}
        \sum_{m=1}^n m \cdot w_{n,m} &= \frac{\sum_{m=1}^n m \cdot \theta_{n,m}}{\theta_n} \\
                                     &= \frac{2^{j_n+1} \cdot \sum_{m=1}^n m \cdot \frac{n}{m} \binom{n+m-1}{2m - 1} 2^{m-1}}{2^{j_n} \cdot \sum_{m=1}^n \frac{n}{m} \binom{n+m-1}{2m - 1} 2^m} \\
                                     &= \frac{T_n'(2)}{T_n (2) - 1} \\
                                     &= \frac{n \cdot U_{n-1}(2)}{T_n (2) - 1} \\
                                     &= \frac{n \cdot [(2 + \sqrt{3})^n - (2 - \sqrt{3})^n]}{2 \sqrt3} \cdot \frac{2}{(2 + \sqrt{3})^n - (2 - \sqrt{3})^n - 1} \smash{\text{\quad\raisebox{3.0\baselineskip}{,}}}\\
                                     &> \frac{n}{\sqrt{3}},
    \end{align*}
    where the fourth equality follows from \eqref{eq:cheb3}, and the fifth equality follows from \eqref{eq:cheb1} and \eqref{eq:cheb2}.

    Therefore,
    \[
        \sigma^2_{X_n} > \frac{n}{4\sqrt{3}},
    \]
    from which follows that $\sigma^2_{X_n} \to \infty$ when $n \to \infty$. So, the coefficients of $\{p_n(t) / p_n(1)\}_{n \geq 1}$ (and, thus, the coefficients of $\{p_n(t)\}_{n \geq 1}$) are asymptotically normal.
\end{proof}

\begin{rmk}
    Set $\theta_{n,0} \coloneqq 2^{j_n}$, $\theta_n \coloneqq \sum_{m=0}^n \theta_{n,m}$ and $w_{n,0} \coloneqq \mbox{\Large $\nicefrac{\theta_{n,m}}{\theta_n}$}$ and $\Tilde{p}_n(t) \coloneqq t^{i_n}(1+t)^{j_n} + p_n(t)$ for each $n \geq 0$. By repeating the arguments of the proof of Corollary \ref{cor:normal-cheb}, we can prove that the coefficients of $\Tilde{p}_n(t)$ (and, in particular, the coefficients of $T_n(1+ \frac{1+t}{2})$) are asymptotically normal.
\end{rmk}

\begin{rmk}
    By \eqref{eq:cheb2},
    \[
        U_n \left(1 + \frac{1+t}{2} \right) = \sum_{m=0}^n \binom{n+m+1}{2m+1} (1+t)^m
    \]
    for $n \geq 2$. Let $\{i_n\}_{n \geq 0}$ and $\{j_n\}_{n \geq 0}$ be sequences of nonnegative integers and set
    \[
        q_n(t) \coloneqq t^{i_n} (1+t)^{j_n} \cdot U_n \left(1 + \frac{1+t}{2} \right)
    \]
    for all $n \geq 0$. Let $X_n$ be the random variable with probability generating polynomial equal to $\mbox{\Large $\nicefrac{q_n(t)}{q_n(1)}$}$. Since
    \[
        U_n'(t) = \frac{(n+1) T_{n+1}(t) - t \cdot U_n(t)}{t^2 - 1}
    \]
    and $U_n(t)$ is real-rooted for all $n \geq 0$ \cite{mason2002chebyshev}, an argument similar to the one used in the proof of Corollary \ref{cor:normal-cheb} shows that
    \[
        \sigma^2_{X_n} \geq \frac{U_n'(2)}{U_n(2)} = \frac{n+1}{4\sqrt{3}} - \frac{1}{6} \quad,
    \]
    from which follows that $\sigma^2_{X_n} \to \infty$ and, hence, the coefficients of $\{q_n(t)\}_{n \geq 0}$ are asymptotically normal.
\end{rmk}

We now prove that the coefficients of $\{ f(\mathcal{O}(\mathcal{C}_{2n}),t) \}_{n \geq 2}$ are asymptotically normal. Different from the polynomials studied in Corollary \ref{cor:normal-cheb}, the sequence of polynomials $\{ f(\mathcal{O}(\mathcal{C}_{2n}),t) \}_{n \geq 2}$ is not real-rooted, so Theorem \ref{th:bender} cannot be applied. However, Lemma \ref{conj-f-vector} shows us that $f(\mathcal{O}(\mathcal{C}_{2n}),t)$ is obtained from the Chebyshev polynomials via a small perturbation.

We prove now that a sequence of probability generating polynomials (not necessarily real-rooted) obtained by a perturbation of a sequence of asymptotically normal probability generating polynomials also satisfies a Central Limit Theorem:

\begin{thm}
\label{th:pert}
    Let $\{X_n\}_{n \geq 0}$ and $\{Y_n\}_{n \geq 0}$ be sequences of random variables taking values in $\{0, 1, \ldots, n\}$, and $\{\gamma_n\}_{n \geq 0}$ and $\{\lambda_n\}_{n \geq 0}$ be sequences of nonnegative real numbers such that
    \[
        \frac{X_n - \mu_{X_n}}{\sigma_{X_n}} \xrightarrow{D} N(0,1), \quad \lambda_n \to 0, \quad \frac{\lambda_n \sigma^2_{Y_n}}{\sigma^2_{X_n}} \to 0 \quad \mbox{and} \quad \frac{\lambda_n (\mu_{Y_n} - \mu_{X_n})^2}{\sigma^2_{X_n}} \to 0 \quad .
    \]
    Let $f_n(t) = p_{X_n}(t)$ and $g_n(t) = p_{Y_n}(t)$ for all $n \geq 0$, and set
    \[
        h_n(t) \coloneqq \frac{\gamma_n f_n(t) + \lambda_n g_n(t)}{\gamma_n + \lambda_n} \quad .
    \]
    Then the coefficients of $\{h_n(t)\}_{n \geq 0}$ are asymptotically normal.
\end{thm}

\begin{proof}
    Without loss of generality, we may assume $\gamma_n = 1$ for every $n \geq 0$. In fact,
    \[
        h_n(t) = \frac{f_n(t) + \mbox{\Large $\nicefrac{\lambda_n}{\gamma_n}$} \cdot g_n(t)}{1 + \mbox{\Large $\nicefrac{\lambda_n}{\gamma_n}$}} \quad .
    \]
    For each $n \geq 0$, let
    \[
        \varepsilon_n = \frac{\lambda_n}{1 + \lambda_n},
    \]
    and let $Z_n$ be the mixture distribution of $X_n$ and $Y_n$ with weights $1 - \varepsilon_n$ and $\varepsilon_n$, respectively. Then $p_{Z_n}(t) = h_n(t) = (1 - \varepsilon_n) f_n(t) + \varepsilon_n g_n(t)$ for all $n \geq 0$. To show that $\{h_n(t)\}_{n \geq 0}$ are asymptotically normal, we will prove that $\{Z_n\}_{n \geq 0}$ satisfies \eqref{eq:asym-normal}.

    In fact, let
    \[
        H_n(t) = \mathbb{P} \left( \frac{Z_n - \mu_{Z_n}}{\sigma_{Z_n}} \leq t \right).
    \]
    Hence, we must show that
    \begin{equation}
    \label{eq:proof}
        \lim_{n \to \infty} \sup_{t \in \mathbb{R}} \left\{ \left| H_n(t) - \frac{1}{2 \sqrt{\pi}} \int_{- \infty}^t e^{-\frac{x^2}{2}} dx \right| \right\} = 0 \quad .
    \end{equation}

    \begin{claim}
    \label{claim:lim1}
        \begin{equation}
        \label{eq:lim1}
            \frac{\sigma_{Z_n}}{\sigma_{X_n}} \to 1.
        \end{equation}
    \end{claim}

    \begin{proof}[Proof of claim]
        By \eqref{eq:mix},
        \[
            \sigma^2_{Z_n} = (1 - \varepsilon_n)\sigma^2_{X_n} + \varepsilon_n \sigma^2_{Y_n} + \varepsilon_n (1 - \varepsilon_n) (\mu_{X_n} - \mu_{Y_n})^2 
        \]
        for every $n \geq 0$, which implies
        \[
            \frac{\sigma_{Z_n}^2}{\sigma_{X_n}^2} = (1 - \varepsilon_n) + \varepsilon_n \frac{\sigma_{Y_n}^2}{\sigma_{X_n}^2} + (1-\varepsilon_n) \frac{\varepsilon_n (\mu_{Y_n} - \mu_{X_n})^2}{\sigma_{X_n}^2}
        \]
        Observe that, since $0 \leq \varepsilon_n \leq \lambda_n$ and $\lambda_n \to 0$, it follows that
        \[
            \varepsilon_n \to 0, \quad \frac{\varepsilon_n \sigma^2_{Y_n}}{\sigma^2_{X_n}} \to 0 \quad \mbox{and} \quad \frac{\varepsilon_n (\mu_{Y_n} - \mu_{X_n})^2}{\sigma^2_{X_n}} \to 0 \quad ,
        \]
        from which the claim follows.
    \end{proof}

    \begin{claim}
    \label{claim:lim2}
        \begin{equation}
        \label{eq:lim2}
            \frac{|\mu_{Z_n} - \mu_{X_n}|}{\sigma_{X_n}} \to 0 \quad .
        \end{equation}
    \end{claim}

    \begin{proof}[Proof of claim]
        By \eqref{eq:mix},
        \[
            \mu_{Z_n} = (1 - \varepsilon_n) \mu_{X_n} + \varepsilon_n \mu_{Y_n} \quad,
        \]
        which implies
        \[
            \frac{\mu_{Z_n} - \mu_{X_n}}{\sigma_{X_n}} = \frac{\varepsilon_n (\mu_{Y_n} - \mu_{X_n})}{\sigma_{X_n}}
        \]
        for every $n \geq 0$. Hence
        \[
           \left( \frac{\mu_{Z_n} - \mu_{X_n}}{\sigma_{X_n}} \right)^2 \leq \lambda_n \frac{\lambda_n (\mu_{Y_n} - \mu_{X_n})^2}{\sigma_{X_n}^2}
        \]
        for every $n \geq 0$. Since
        \[
            \lambda_n \to 0 \quad \mbox{and} \quad \frac{\lambda_n (\mu_{Y_n} - \mu_{X_n})^2}{\sigma^2_{X_n}} \to 0 \quad ,
        \]
        the claim follows.
    \end{proof}

    Now we can prove \eqref{eq:proof}. For each $n \geq 0$, set
    \[
        F_n(t) \coloneqq \mathbb{P} \left(  \frac{X_n - \mu_{X_n}}{\sigma_{X_n}} \leq t \right) \quad.
    \]
    Hence,
    \begin{align*}
        H_n(t) & = (1 - \varepsilon_n) \mathbb{P}\left(  X_n  \leq \sigma_{Z_n} \cdot t + \mu_{Z_n} \right) + \varepsilon_n \mathbb{P}\left(  Y_n  \leq \sigma_{Z_n} \cdot t + \mu_{Z_n} \right) \\
               & \leq (1 - \varepsilon_n) F_n \left( \frac{\sigma_{Z_n}}{\sigma_{X_n}} \cdot t + \frac{\mu_{Z_n} - \mu_{X_n}}{\sigma_{X_n}} \right) + \varepsilon_n
    \end{align*}
    for all $n \geq 0$. Hence
    \[
        \left| H_n(t) - \frac{1}{2 \sqrt{\pi}} \int_{- \infty}^t e^{-\frac{x^2}{2}} dx \right| \leq \left| \left[ (1 - \varepsilon_n) F_n \left( \frac{\sigma_{Z_n}}{\sigma_{X_n}} \cdot t + \frac{\mu_{Z_n} - \mu_{X_n}}{\sigma_{X_n}} \right) - \frac{1}{2 \sqrt{\pi}} \int_{- \infty}^t e^{-\frac{x^2}{2}} dx \right] + \varepsilon_n \right|
    \]
    By Claims \ref{claim:lim1} and \ref{claim:lim2},
    \[
        \frac{\sigma_{Z_n}}{\sigma_{X_n}} \to 1 \quad \mbox{and} \quad \frac{\mu_{Z_n} - \mu_{X_n}}{\sigma_{X_n}} \to 0 \quad ,
    \]
    from which follows that
    \[
        F_n \left( \frac{\sigma_{Z_n}}{\sigma_{X_n}} \cdot t + \frac{\mu_{Z_n} - \mu_{X_n}}{\sigma_{X_n}} \right) \to \frac{1}{2 \sqrt{\pi}} \int_{- \infty}^t e^{-\frac{x^2}{2}} dx
    \]
    uniformly. Since $\varepsilon_n \to 0$, we conclude
    \[
        \sup_{t \in \mathbb{R}} \left\{ \left| H_n(t) - \frac{1}{2 \sqrt{\pi}} \int_{- \infty}^t e^{-\frac{x^2}{2}} dx \right| \right\} \to 0 \quad,
    \]
    which proves the theorem.
\end{proof}

\begin{cor}
\label{cor:pert}
    Let $\{X_n\}_{n \geq 0}$ be a sequence of random variables taking values in $\{0, 1, \ldots, n\}$ such that
    \[
        \quad \frac{X_n - \mu_{X_n}}{\sigma_{X_n}} \xrightarrow{D} N(0,1) \quad .
    \]
    Let $f_n(t) = p_{X_n}(t)$ for all $n \geq 0$, and let $\{g_n(t)\}_{n \geq 0}$ be a sequence of polynomials with nonnegative coefficients such that $g_n(1) = 1$ and $\deg(g_n(t)) \leq n$ for every $n \geq 0$. Let $\{\gamma_n\}_{n \geq 0}$ and $\{\lambda_n\}_{n \geq 0}$ be sequences of nonnegative real numbers, and set
    \[
        h_n(t) \coloneqq \frac{\gamma_n f_n(t) + \lambda_n g_n(t)}{\gamma_n + \lambda_n} \quad .
    \]
    If
    \[
        \frac{\lambda_n}{\gamma_n} \frac{n^2}{\sigma^2_{X_n}} \to 0 \quad,
    \]
    then the coefficients of $\{h_n(t)\}_{n \geq 0}$ are asymptotically normal.
\end{cor}

\begin{proof}
    Again, we may assume $\gamma_n = 1$. For each $n \geq 0$, let $Y_n$ be the random variable whose probability generating polynomial is $g_n(t)$. By Theorem \ref{th:pert}, it is sufficient to show that
    \[
        \lambda_n \to 0 \quad \mbox{and} \quad \frac{\lambda_n \sigma^2_{Y_n}}{\sigma^2_{X_n}} \to 0 \quad .
    \]
    
    In fact, let
    \[
        \varepsilon_n = \frac{\lambda_n}{1 + \lambda_n}.
    \]
    By Popoviciu's inequality on variances \cite{popoviciu1933}, $\sigma^2_{X_n} \leq \mbox{\Large $\nicefrac{n^2}{4}$}$, which implies
    \[
        0 \leq 4 \lambda_n \leq \frac{\lambda_n n^2}{\sigma^2_{X_n}} \quad .
    \]
    Since $\mbox{\Large $\nicefrac{\lambda_n n^2}{\sigma^2_{X_n}}$} \to 0$ by hypothesis, it follows that $\lambda_n \to 0$.

    Again, by Popoviciu's inequality on variances, $\sigma^2_{Y_n} \leq \mbox{\Large $\nicefrac{n^2}{4}$}$, which implies
    \[
        0 \leq \frac{\lambda_n \sigma^2_{Y_n}}{\sigma^2_{X_n}} \leq \frac{1}{4} \frac{\lambda_n n^2}{\sigma^2_{X_n}}
    \]
    and since $\mbox{\Large $\nicefrac{\lambda_n n^2}{\sigma^2_{X_n}}$} \to 0$, it follows that $\mbox{\Large $\nicefrac{\lambda_n \sigma^2_{Y_n}}{\sigma^2_{X_n}}$} \to 0$, which concludes the proof of the corollary.
\end{proof}

\begin{rmk}
    Theorem \ref{th:pert} and Corollary \ref{cor:pert} can be easily generalized to the case where
    \[
        \frac{X_n - \mu_{X_n}}{\sigma_{X_n}} \xrightarrow{D} X,
    \]
    where $X$ is a random variable with cumulative distribution function $F(t)$.
\end{rmk}

\begin{cor}
\label{cor:asymp-normal}
    Let $\{\alpha_n\}_{n \geq 0}$ and $\{ \beta_n\}_{n \geq 0}$ be sequences of nonnegative real numbers, and $\{i_n\}_{n \geq 0}$, $\{j_n\}_{n \geq 0}$, $\{k_n\}_{n \geq 0}$ and $\{l_n\}_{n \geq 0}$ be sequences of nonnegative integer numbers such that $k_n + l_n \leq n + i_n + j_n$ for every $n \geq 0$ and
    \[
        2^{i_n - k_n} (n + i_n + j_n)^2 \cdot \frac{\beta_n}{\alpha_n} \left( \frac{2}{2 + \sqrt{3}} \right)^n \to 0 \quad.
    \]
    Set
    \[
        c_n(t) = \beta_n t^{k_n} (1+t)^{l_n} + \alpha_n t^{i_n} (1+t)^{j_n} \sum_{m=1}^n \frac{n}{m} \binom{n+m-1}{2m - 1} (1+t)^m
    \]
    and
    \[
        d_n(t) = \beta_n t^{k_n} (1+t)^{l_n} + \alpha_n t^{i_n} (1+t)^{j_n} \cdot U_n \left(1 + \frac{1+t}{2} \right) \quad .
    \]
    Then the coefficients of $\{c_n(t)\}_{n \geq 1}$ and $\{d_n(t)\}_{n \geq 0}$ are asymptotically normal. In particular, the coefficients of $\{ f(\mathcal{O}(\mathcal{C}_{2n}),t) \}_{n \geq 2}$ are asymptotically normal.
\end{cor}

\begin{proof}
    For each $n \geq 1$, set
    \[
        a_n(t) \coloneqq \alpha_n t^{i_n} (1+t)^{j_n} \sum_{m=1}^n \frac{n}{m} \binom{n+m-1}{2m - 1} (1+t)^m \quad \mbox{and} \quad b_n(t) \coloneqq \beta_n t^{k_n} (1+t)^{l_n} \quad,
    \]
    and let
    \[
        \gamma_n = a_n(1), \quad \lambda_n = b_n(1), \quad f_n(t) = \frac{a_n(t)}{a_n(1)} \quad \mbox{and} \quad g_n(t) = \frac{b_n(t)}{b_n(1)} \quad .
    \]
    For each $n \geq 1$, define
    \[
        h_n(t) \coloneqq \frac{c_n(t)}{c_n(1)} = \frac{\gamma_n f_n(t) + \lambda_n g_n(t)}{\gamma_n + \lambda_n}
    \]
    and let $X_n$ be a random variable whose probability generating polynomial is $f_n(t)$. Since the coefficients of $\{f_n(t)\}_{n \geq 1}$ are asymptotically normal by Corollary \ref{cor:normal-cheb}, it is sufficient to prove that
    \[
        \frac{\lambda_n}{\gamma_n} \frac{(n + i_n + j_n)^2}{\sigma^2_{X_n}} \to 0 \quad.
    \]
    In fact, by definition,
    \[
        \lambda_n = 2^{l_n} \beta_n  \quad \mbox{and} \quad \gamma_n = 2^{j_n} \alpha_n \left[T_n(2) - 1 \right] = 2^{j_n - 1 } \alpha_n \cdot [(2 + \sqrt{3})^n + (2 - \sqrt{3})^n - 1]
    \]
    and, by the proof of Corollary \ref{cor:normal-cheb},
    \[
        \sigma^2_{X_n} > \frac{n}{4 \sqrt{3}} \quad .
    \]
    Hence,
    \begin{align*}
        \frac{\lambda_n}{\gamma_n} \frac{(n + i_n + j_n)^2}{\sigma^2_{X_n}} & < \frac{2^{i_n - k_n +3 } \sqrt{3} (n + i_n + j_n)^2}{1 + (7-4\sqrt{3})^n} \frac{\beta_n}{\alpha_n} \left( \frac{2}{2 + \sqrt{3}} \right)^n \smash{\text{\qquad\raisebox{-0.8\baselineskip}{,}}} \\
        & < 2^{i_n - k_n + 3} \sqrt{3} (n + i_n + j_n)^2 \cdot \frac{\beta_n}{\alpha_n} \left( \frac{2}{2 + \sqrt{3}} \right)^n
    \end{align*}
    which goes to $0$ as $n \to \infty$, by hypothesis, and proves the asymptotically normality of the coefficients of $\{c_n(t)\}_{n \geq 1}$. Analogously, we prove the asymptotically normality of the coefficients of $\{d_n(t)\}_{n \geq 0}$.

    To verify that the coefficients of $\{ f(\mathcal{O}(\mathcal{C}_{2n}),t) \}_{n \geq 2}$ are asymptotically normal, take  $i_n = 1$, $j_n = n$, $k_n = 0$, $l_n = 2$ and $\alpha_n = \beta_n = 1$ for every $n \geq 0$ in the sequence $\{c_n(t)\}_{n \geq 1}$.
\end{proof}

\section{The Ehrhart polynomial of \texorpdfstring{$\mathcal{O}(\crown_{2n})$}-}
\label{sec-ehrhart}
In this section, we study the order polynomial of $\mathcal{C}_{2n}$. We present a recursive formula to compute $\Omega_{\mathcal{C}_{2n}}(t)$ in terms of the order polynomials of zigzags. We use this recursion to give an alternative proof for the nonnegativity of the coefficients of $\Omega_{\mathcal{C}_{2n}}(t)$, recently proved by Ferroni, Morales and Panova \cite{ferroni2025skew}.

Given a finite poset $P$, let $\Omega_P \colon \mathbb{R} \to \mathbb{R}$ be the function such that $\Omega_P (n)$ counts the number of order-preserving maps from $P$ to $[n]$. This function was first introduced by Richard P. Stanley in his PhD thesis, where it was also proved that it is a polynomial \cite[Proposition 13.1]{stanleyPHDthesis}, called the \emph{order polynomial} of $P$. Later, Stanley proved \cite[Theorem 4.1]{stanley1986two} that the Ehrhart polynomial $L(\mathcal{O}(P),t)$ of $\mathcal{O}(P)$ and the order polynomial $\Omega_P(t)$ of $P$ are related by the following:
\[
    L(\mathcal{O}(P),t) = \Omega_P(t+1).
\]
Therefore, the results in this section also provide a formula to compute the Ehrhart polynomial of $\mathcal{O}(\mathcal{C}_{2n})$ as an expression of Ehrhart polynomials of zigzags and prove that its coefficients are nonnegative.

We start by stating the reciprocity theorem for graded posets. Recall that a poset $P = (P, \preceq)$ is \emph{graded} if there exists a rank function $\rho \colon P \to \mathbb{N}$ such that
\begin{enumerate}[(i)]
    \item $x \prec y$ implies $\rho(x) < \rho(y)$;
    \item if $y$ covers $x$, then $\rho(y) = \rho(x) + 1$.
\end{enumerate}
A graded poset $P$ is called $r$\emph{-graded} if every maximal chain of $P$ has length $r$. For example, $\mathcal{C}_{2n}$ and $Z_{n+1}$ are $1$-graded posets for every $n \geq 1$.

\begin{lem}\cite[Theorem 2.3]{MR2120105}
\label{reciprocity-graded-posets}
    Let $P = ([n], \preceq)$ be an $r$-graded poset. Then
    \[
        \Omega_P(t-r) = (-1)^{n} \Omega_P(-t).
    \]
\end{lem}

Now we can state the main result of our section. In the following proof, for $a,b \in \mathbb{Z}$ we shall write $[a,b]$ for the closed interval $\{ k \in \mathbb{Z} \colon a \le k \le b \}$.

\begin{thm}
\label{ehrhart-recursion}
    \begin{align}
        \Omega_{\mathcal{C}_{2n}}(t) &= \Omega_{\mathcal{C}_{2n}}(t-1) + n \cdot \Omega_{Z_{2n-1}}(t-1) + \Omega_{Z_{2n-3}}(t) + \sum_{1 \leq i < j \leq n} \Omega_{Z_{2(n-j+i)-1}}(t-1) \cdot \Omega_{Z_{2(j-i)-1}}(t) \label{recursion-order-pol} \\
        &= \Omega_{\mathcal{C}_{2n}}(-t) - n \cdot \Omega_{Z_{2n-1}}(-t) + t \cdot \Omega_{Z_{2n-3}}(t) - \sum_{\substack{1 \leq i < j \leq n \\ (i,j) \neq (1,n)}} \Omega_{Z_{2(n-j+i)-1}}(-t) \cdot \Omega_{Z_{2(j-i)-1}}(t) \label{recursion-order-pol2}
    \end{align}
    for every $n \geq 2$.
\end{thm}

\begin{proof}
    Since Equation~\ref{recursion-order-pol} is an equality between polynomials, it is sufficient to prove it for nonnegative integers. Let $m \geq 2$ be an integer. Let $\mathcal{F}(\mathcal{C}_{2n}, m)$ be the set of order-preserving maps from $\mathcal{C}_{2n}$ to $[m]$. By definition, $|\mathcal{F}(\mathcal{C}_{2n}, m)| = \Omega_{\mathcal{C}_{2n}} (m)$. We prove \eqref{recursion-order-pol} by partitioning the set $\mathcal{F}(\mathcal{C}_{2n}, m)$.

    Recall that $\mathcal{C}_{2n}$ is defined by the cover relations $1 \prec 2 \succ 3 \prec \cdots \succ 2n-1 \prec 2n \succ 1$, see Figure~\ref{fig-crown-poset}.
    Let
    \begin{gather*}
        \mathcal{F}_1 \coloneqq \{ f \in \mathcal{F}(\mathcal{C}_{2n}, m) \colon f(a) < m \mbox{ for all } a \in [2n] \} \smash{\text{\quad\raisebox{-1.0\baselineskip}{.}}}\\
        \mathcal{F}_2 \coloneqq \{ f \in \mathcal{F}(\mathcal{C}_{2n}, m) \colon |f^{-1}(m)| = 1  \} \\
        \mathcal{F}_3 \coloneqq \{ f \in \mathcal{F}(\mathcal{C}_{2n}, m) \colon |f^{-1}(m)| \geq 2  \}
    \end{gather*}
    Clearly, $\mathcal{F}(\mathcal{C}_{2n}, m) = \mathcal{F}_1 \sqcup \mathcal{F}_2 \sqcup \mathcal{F}_3$. Now, let us partition $\mathcal{F}_3$ as $\mathcal{F}_4 \sqcup \mathcal{F}_5$, where
    \[
        \mathcal{F}_4 \coloneqq \{ f \in \mathcal{F}_3 \colon f(1) = m \} \quad \mbox{and} \quad \mathcal{F}_5 \coloneqq \{ f \in \mathcal{F}_3 \colon f(1) < m \}.
    \]
    We claim that
    \begin{equation}
    \label{eq-disjoint-union}
        \mathcal{F}_5 = \bigsqcup_{1 \leq i < j \leq n} \mathcal{F}_{i,j},
    \end{equation}
    where $\mathcal{F}_{i,j} = \{ f \in \mathcal{F}_5 \colon f(2i) = f(2j) = m \mbox{ and } f^{-1}(m) \subseteq [2i,2j] \}.$ In fact, let $f \in \mathcal{F}_5$. Then $f(1) < m$ and there exist $a < b$ in $[2n]$ such that $f(a) = f(b) = m$. Assume, without loss of generality, that $a = \min \{ k \in [2n] \colon f(k) = m \}$ and $b = \max \{ k \in [2n] \colon f(k) = m \}$. By definition, $a$ and $b$ must be even, that is, $a = 2i$ and $b = 2j$ for some $i < j$ in $[n]$. Hence, $f \in \mathcal{F}_{i,j}$. Moreover, $\mathcal{F}_{i,j} \cap \mathcal{F}_{k,l} = \emptyset$ if $(i,j) \neq (k,l)$, which verifies \eqref{eq-disjoint-union}.

    Now, we can finally prove \eqref{recursion-order-pol}. By definition, $| \mathcal{F}_1| = \Omega_{\mathcal{C}_{2n}}(m-1)$. To count the number of elements in $\mathcal{F}_2$, let $f \in \mathcal{F}_2$ be such that $f(a) = m$ for some $a \in [2n]$. By definition of $\mathcal{F}_2$, $a$ must be even, and there are $\Omega_{Z_{2n-1}}(m-1)$ ways of assigning values to the other $2n-1$ elements of $\mathcal{C}_{2n}$. Since there are $n$ even numbers in $[2n]$, $|\mathcal{F}_2| = n \cdot \Omega_{Z_{2n-1}}(m-1)$. To count the number of elements in $\mathcal{F}_4$, observe that
    \[
        \mathcal{F}_4 = \{ f \in \mathcal{F}(\mathcal{C}_{2n}, m) \colon f(1) = f(2) = f(2n) = m\}.
    \]
    Hence $|\mathcal{F}_4| = \Omega_{Z_{2n-3}}(m)$. Lastly, to count the number of elements of $\mathcal{F}_5$, let $i < j$ be numbers in $[n]$ and let $f \in \mathcal{F}_{i,j}$. There are $2j - 2i - 1$ elements between $2i$ and $2j$, and these elements can be mapped to $m$ by $f$, so there are $\Omega_{Z_{2(j-i)-1}}(m)$ ways to assign values to $f(2i+1), \ldots, f(2j-1)$. Meanwhile, $1, \ldots, 2i-1, 2j+1, \ldots, 2n$ cannot be mapped to $m$ by $f$, hence there are $\Omega_{Z_{2(n-j+i)-1}}(m-1)$ ways to assign values of $f$ to these elements. The choices between $2i$ and $2j$, and to the left of $2i$ or to the right of $2j$ are independent of each other, so $|\mathcal{F}_{i,j}| = \Omega_{Z_{2(n-j+i)-1}}(m-1) \cdot \Omega_{Z_{2(j-i)-1}}(m)$. Summing over all $1 \leq i < j \leq n$ gives the final term in \eqref{recursion-order-pol}.

    To prove \eqref{recursion-order-pol2}, first observe that, if $1 \leq i < j \leq n$, then $2(n-j+i)-1 = 1$ if and only if $i = 1$ and $j=n$. We rewrite Equation \ref{recursion-order-pol} as
    \begin{align}
        \Omega_{\mathcal{C}_{2n}}(t) = &\Omega_{\mathcal{C}_{2n}}(t-1) + n \cdot \Omega_{Z_{2n-1}}(t-1) + \Omega_{Z_{2n-3}}(t) + \Omega_{Z_1}(t-1) \cdot \Omega_{Z_{2n-3}}(t) \nonumber \\
        &+ \sum_{\substack{1 \leq i < j \leq n \\ (i,j) \neq (1,n)}} \Omega_{Z_{2(n-j+i)-1}}(t-1) \cdot \Omega_{Z_{2(j-i)-1}}(t) \nonumber \\
        = &\Omega_{\mathcal{C}_{2n}}(t-1) + n \cdot \Omega_{Z_{2n-1}}(t-1) + t \cdot \Omega_{Z_{2n-3}}(t) + \sum_{\substack{1 \leq i < j \leq n \\ (i,j) \neq (1,n)}} \Omega_{Z_{2(n-j+i)-1}}(t-1) \cdot \Omega_{Z_{2(j-i)-1}}(t) \label{recursion-order-pol3}
    \end{align}
    where the second equality holds because $\Omega_{Z_1}(t) = t$. Equation \ref{recursion-order-pol2} follows from a combination of \eqref{recursion-order-pol3} and Lemma \ref{reciprocity-graded-posets}.
\end{proof}

\begin{rmk}
    Using a similar argument as the one used in Theorem \ref{ehrhart-recursion}, one can show that
    \[
        \Omega_{Z_n}(t) = \Omega_{Z_n}(t-1) + \Omega_{Z_{n-2}}(t) + \sum_{1 \leq i \leq j \leq \lfloor n/2 \rfloor} \Omega_{Z_{2i-1}}(t-1)\cdot \Omega_{Z_{2j-2i-1}}(t) \cdot \Omega_{Z_{n-2j}}(t-1),
    \]
    where we set $\Omega_{Z_{-1}}(t) \coloneqq 1$, which is a recursive formula that does not seem to appear in the literature.
\end{rmk}

Given a polynomial $p(t) \in \mathbb{R}[t]$, let $[t^k] p(t)$ denote the coefficient of $t^k$ in $p(t)$.

\begin{lem}[\cite{ferroni2025skew}, Page 10] \label{linear-coefficient-zigzag}
    For every $n \ge 1$ the linear coefficient of $\Omega_{Z_{n}}(t)$ is given by 
    \[
        [t]\Omega_{Z_{n}}(t) = \frac{\lfloor \frac{n-1}{2} \rfloor ! \cdot \lceil \frac{n-1}{2} \rceil !}{n!} = \frac{1}{n} \binom{n-1}{\lfloor\frac{n-1}{2} \rfloor}^{-1} \quad.
    \]
\end{lem}

In order to prove that the order polynomial of $\mathcal{C}_{2n}$ has nonnegative coefficients, the crucial difficulty is showing that the linear term is indeed nonnegative. Unlike the case of zigzags, which come from a skew shape, the computation of the linear term of the crown (which falls into the class of \emph{cylindric skew shapes}) is more subtle and requires a complicated determinantal formula (see \cite[Section 7]{ferroni2025skew}). By using our recursions, we can give a simpler proof of this nonnegativity:

\begin{cor} \label{linear-coefficient}
    For every $n \ge 1$ the linear coefficient of $\Omega_{\crown_{2n}}(t)$ is given by 
    \begin{equation}
    \label{eq-linear-coefficient}
        [t] \Omega_{\mathcal{C}_{2n}}(t) = \binom{2n}{n}^{-1} \quad .
    \end{equation}
\end{cor}
\begin{proof}
    If $n = 1$, then $\Omega_{\mathcal{C}_{2}} (t) = t/2 + t^2/2$. Recall that $\Omega_P(0) = 0$ if $P$ is a non-empty poset, that is, $[1] \Omega_P(t) = 0$ if $P \neq \emptyset$. Observe that $[t]\Omega_{\mathcal{C}_{2n}}(-t) = -[t]\Omega_{\mathcal{C}_{2n}}(t)$. Then, by \eqref{recursion-order-pol2}, if $n\geq 2$,
    \begin{align*}
        [t]\Omega_{\mathcal{C}_{2n}}(t) &= [t]\Omega_{\mathcal{C}_{2n}}(-t) - n \cdot [t]\Omega_{Z_{2n-1}}(-t) \implies \\
        [t]\Omega_{\mathcal{C}_{2n}}(t) &= \frac{n}{2} \cdot \frac{1}{(2n-1) \binom{2n-2}{n-1}} = \binom{2n}{n}^{-1}
    \end{align*}
    where Lemma \ref{linear-coefficient-zigzag} was used in the second equality.
\end{proof}

Let $P$ be a poset. An \emph{order ideal} of $P$ is a subset $I$ of $P$ such that if $t \in I$ and $s \leq t$ in $P$, then $s \in I$. Denote by $J(P)$ the distributive lattice of order ideals in $P$ ordered by inclusion and, if $P$ is finite, define the polynomial
\[
    c_{J(P)}(t) \coloneqq \sum_{k \geq 0} c_k (J(P)) t^k,
\]
where $c_k(J(P)) = |\{ \emptyset = I_0 < I_1 < \cdots < I_k = P \}|$. The polynomial $c_{J(P)}(t)$ is known as the \emph{chain polynomial} of $J(P)$.

Set 
\[
    \phi(P) = \sum_{k=1}^{|P|} (-1)^{k-1} \frac{c_k (J(P))}{k}.
\]

The following lemma is a combinatorial description of the coefficients of the order polynomial of a poset.

\begin{lem}\cite[Proposition 2.10]{shareshian2003new}
\label{shareshian-lemma}
    Let $P = ([n], \preceq)$ be a poset. Then
    \[
        [t^k]\Omega_P(t) = \frac{1}{k!} \cdot \sum_{\emptyset \subsetneq I_1 \subsetneq I_2 \subsetneq \cdots \subsetneq I_k = P} \phi(I_1) \cdot \phi(I_2 \setminus I_1) \cdot \ldots \cdot \phi(I_k \setminus I_{k-1}).
    \]
    In particular, $[t]\Omega_P(t) = \phi(P).$
\end{lem}

The importance of Lemma \ref{shareshian-lemma} is the following: given a finite poset $P$, if $\phi(I) \geq 0$ for every order ideal $\emptyset \neq I \subseteq P$ and $\phi(I \setminus J) \geq 0$ for every pair $\emptyset \subsetneq J \subsetneq I \subseteq P$, then the coefficients of $\Omega_P(t)$ are nonnegative. This observation was recently rediscovered by Ferroni, Morales and Panova \cite{ferroni2025skew} as their Meta Theorem.

The next lemma is well-known and the proof is omitted:
\begin{lem}
\label{disjoint-union-posets-order-pol}
    Let $P$ and $Q$ be disjoint finite posets. Then
    \[
        \Omega_{P \sqcup Q}(t) = \Omega_P(t) \cdot \Omega_Q(t).
    \]
\end{lem}

Now, we have all the ingredients to prove the nonnegativity of the coefficients of $\Omega_{\mathcal{C}_{2n}}(t)$:

\begin{cor}
    $[t^k]\Omega_{\mathcal{C}_{2n}}(t) \geq 0$ for every $k \geq 0$, and every $n \geq 1$.
\end{cor}

\begin{proof}
    By definition of the order polynomial, if $P \neq \emptyset$, then $[1] \Omega_P(t) = 0$, and by Corollary \ref{linear-coefficient}, $[t]\Omega_{\mathcal{C}_{2n}}(t) \geq 0$. For $k \geq 2$, the result follows from the fact that any order ideal $I \not\in \{\emptyset, \mathcal{C}_{2n}\}$ of $\mathcal{C}_{2n}$ is a disjoint union of zigzags, just like $J \setminus I$ for every pair $\emptyset \subsetneq I \subsetneq J \subseteq \mathcal{C}_{2n}$.
\end{proof}

\section{The \texorpdfstring{$h^*$}--polynomial of \texorpdfstring{$\mathcal{O}(\crown_{2n})$}-}
\label{sec-h-vector}
In this section, we study the $h^*$-polynomial of the order polytope $\mathcal{O}(\mathcal{C}_{2n})$. We give a combinatorial interpretation for the coefficients of the $h^*$-polynomial of $\mathcal{O}(\mathcal{C}_{2n})$. To do this, we study a subset of alternating permutations, called cyclically alternating permutations. We will adapt the methods in \cite{coons2019h} to the cyclically alternating case.

Let $\sigma = \sigma_1 \sigma_2 \cdots \sigma_{2n}\in S_{2n}$ be a permutation. We say that $\sigma$ is a \emph{cyclically alternating permutation} if $\sigma_1 < \sigma_2 > \sigma_3 < \ldots > \sigma_{2n - 1} < \sigma_{2n} > \sigma_1$. The set of cyclically alternating permutations on $[2n]$ is denoted by $\operatorname{CA_{2n}}$, and we set $A_\circ (2n) = |\operatorname{CA_{2n}}|$. For more about cyclically alternating permutations, we refer to \cite{elkies2003sums}. For ease of exposition, we will sometimes refer to cyclically alternating permutations as CAPs.

\begin{defi}
For a CAP $\sigma$, we define the permutation statistic $\operatorname{cswap}(\sigma)$ to be the number of $i < 2n$ such that $\sigma^{-1}(i) > \sigma^{-1}(i+1) + 1$. Equivalently, it is the number of $i < 2n$ such that $i$ is to the right of $i + 1$ in $\sigma$ and swapping $i$ and $i+1$ in $\sigma$ yields a CAP.
\end{defi}

We say that a CAP $\sigma$ \emph{cyclically swaps} to another CAP $\tau$ when $\tau = \sigma\cdot (i, i+1)$ for some $i$ such that $i$ is to the right of $i+1$ in $\sigma$. Note that $\tau$ has one less inversion than $\sigma$. The cyclic swap set $\operatorname{CSwap}(\sigma)$ is the set of $i < 2n$ for which we can perform a cyclic swap in $\sigma$.

\begin{exam}
    For $n=2$, then $\sigma = 2413$ is a CAP, $\sigma^{-1} = 3142$, $\operatorname{cswap}(\sigma) = 2$, and $\operatorname{CSwap}(\sigma) = \{ 1, 3 \}$.
\end{exam}

\begin{rmk}\label{rmk:cyclicswap}
The notion of cyclically swapping is the reverse notion of swapping as introduced in \cite{coons2019h}. That is, if a CAP $\sigma$ cyclically swaps to $\tau$, then $\tau$ swaps to $\sigma$ in the language of \cite{coons2019h}. 
\end{rmk}

Given a permutation $\sigma \in S_n$, $\Delta^\sigma$ denotes the simplex given by the convex hull of $\operatorname{vert}(\sigma) \coloneqq \{v_0^{\sigma}, v_1^{\sigma}, \ldots, v_n^{\sigma}\}$, where $v_0^\sigma = \sum_{i=1}^n e_i$ and $v_i^{\sigma}$ is recursively defined to be $v_i^{\sigma} = v_{i-1}^\sigma - e_{\sigma^{-1}(i)}$.

\begin{prop}[Proposition $2.1$ \cite{coons2019h}]\label{prop:intinfacet}
Let $\sigma$ and $\tau$ be alternating permutations. Then $\Delta^\sigma$ intersects $\Delta^\tau$ in a facet of $\Delta^\sigma$ if and only if $\sigma$ swaps to $\tau$ or $\tau$ swaps to $\sigma$. 
\end{prop}
As a result of Proposition~\ref{prop:intinfacet} and Remark~\ref{rmk:cyclicswap}, we see that, for two CAPs $\sigma$ and $\tau$, we have $\Delta^\sigma$ intersects $\Delta^\tau$ in a facet of $\Delta^\sigma$ if and only if $\sigma$ cyclically swaps to $\tau$ or vice versa. \par
Let $i < j$. We call $(i, j)$ an \emph{inversion} of a permutation $\sigma$ if $\sigma_i > \sigma_j$. The \emph{inversion number} of $\sigma$, denoted $\operatorname{inv}(\sigma)$, is the number of inversions $(i, j)$ of $\sigma$. Similarly, a $\emph{non-inversion}$ of $\sigma$ is a pair $(i, j)$ with $i < j$ such that $\sigma_i < \sigma_j$. The pair $(i,j)$ is called \emph{relevant} if $i < j - 1$.
\begin{lem}\label{lem:swapexists}
Let $\sigma$ be a CAP. If $(\sigma^{-1}(a), \sigma^{-1}(b))$ is a relevant inversion of $\sigma$, then there exists $k$ with $b \leq k < a$ such that $k \in \operatorname{CSwap}(\sigma)$. 
\end{lem}
\begin{proof}
We proceed by induction on $a-b$. 
If $a-b=1$, then $\sigma^{-1}(a) = \sigma^{-1}(b+1) < \sigma^{-1}(b) -1 $ gives $b \in \operatorname{CSwap}(\sigma)$.
So, let us assume the result true for $a-b \geq 1$ and let us prove it for $a - b + 1$. We look at $b+1$ in $\sigma$. There are three possible cases: 
\begin{enumerate}
\item if $\sigma^{-1}(a) < \sigma^{-1}(b+1) -1$, then $(\sigma^{-1}(a), \sigma^{-1}(b+1))$ is a relevant inversion (recall that $a > b + 1$). Since $a-(b+1)< a-b$, there exists $b \leq b + 1 \leq k < a$ such that $k \in \operatorname{CSwap}(\sigma)$ by induction; 
\item if $\sigma^{-1}(a) > \sigma^{-1}(b+1)-1$, then we claim that $\sigma^{-1}(a) > \sigma^{-1}(b+1)$. In fact, if $\sigma^{-1}(a) = \sigma^{-1}(b+1)$, then $a = b + 1$ , but, by hypothesis, $a > b + 1$. So, $\sigma^{-1}(a) > \sigma^{-1}(b+1)$, which implies $\sigma^{-1}(b+1) + 1 \leq \sigma^{-1}(a) < \sigma^{-1}(b)$ and gives $b \in \operatorname{CSwap}(\sigma)$;
\item if $\sigma^{-1}(a) = \sigma^{-1}(b+1)-1$, then we claim that $\sigma^{-1}(b) > \sigma^{-1}(b+1) +1$. In fact, if $\sigma^{-1}(b) \leq \sigma^{-1}(b+1) +1$, then we have three cases:
\begin{itemize}
    \item if $\sigma^{-1}(b) = \sigma^{-1}(b+1)$, then $b = b+1$, since $\sigma$ is a bijection, so this case cannot happen;
    \item if $\sigma^{-1}(b) = \sigma^{-1}(b+1) +1$, then $\sigma^{-1}(a) < \sigma^{-1}(b+1) < \sigma^{-1}(b)$, but $\sigma^{-1}(a) = \sigma^{-1}(b+1)-1$ and $\sigma^{-1}(b) = \sigma^{-1}(b+1) +1$ imply that $a, b+1, b$ is a decreasing consecutive sequence in $\sigma$, which contradicts that $\sigma$ is a CAP;
    \item if $\sigma^{-1}(b) < \sigma^{-1}(b+1)$, then $\sigma^{-1}(a) < \sigma^{-1}(b) < \sigma^{-1}(b+1)$, but $\sigma^{-1}(a) = \sigma^{-1}(b+1)-1$ implies that $\sigma^{-1}(a) < \sigma^{-1}(b) < \sigma^{-1}(a) +1$, which is a contradiction.
\end{itemize}
Hence, $\sigma^{-1}(b) > \sigma^{-1}(b+1) +1$ and it gives $b \in \operatorname{CSwap}(\sigma)$, as needed.
\end{enumerate}
\end{proof}

Let $\sigma$ be a cyclically alternating permutation. We define 
$$E(\sigma) := \{v \colon v \in  \vrt(\sigma) \setminus \vrt(\tau) \text{ for some CAP } \tau \text{ that } \sigma \text{ cyclically swaps to} \}.$$

\begin{prop}\label{prop:mininversion}
A cyclically alternating permutation $\sigma$ minimizes the inversion number of all cyclically alternating permutations $\tau$ with $E(\sigma) \subseteq \operatorname{vert}(\tau)$. 
\end{prop}
\begin{proof}
Let $v_k^{\sigma} \in \operatorname{vert}(\sigma)$. By definition, $v_k^{\sigma}$ is the vector $(1, \ldots, 1)$ with a $1$ removed from the coordinate at the $\sigma^{-1}(\ell)$th position for all $\ell \leq k$. We can read off inversions $(i, j)$ of $\sigma$ with $\sigma(j) \leq k < \sigma(i)$ from $v_k^{\sigma}$ as pairs where the $i$th component of $v_k^{\sigma}$ is $1$ and the $j$th component is $0$. \par
Now we claim that every relevant inversion of $\sigma$ can be read from an element of $E(\sigma)$ as above. Given a relevant inversion $(i, j)$ of $\sigma$, we know, by Lemma~\ref{lem:swapexists}, that there exists a cyclic swap $k$ with $\sigma(j) \leq k < \sigma(i)$ and $(i, j)$ can be read from $v_k^{\sigma}$. Since $k$ is a cyclic swap, $v_k^{\sigma}$ is the opposite vertex of the intersection of $\Delta^\sigma$ and $\Delta^\tau$ for the CAP $\tau$ that $\sigma$ cyclically swaps to when swapping $k$ and $(k + 1)$. Therefore, all relevant inversions of $\sigma$ can be found as non-adjacent $1-0$ pairs in a vertex in $E(\sigma)$. This means we can count the number of relevant inversions of $\sigma$ from vertices in $E(\sigma)$. If $E(\sigma) \subseteq \operatorname{vert}(\tau)$, then all inversions of $\sigma$ are inversions in $\tau$, though $\tau$ could have more. Therefore, $\sigma$ minimizes inversions over all CAP $\tau$ with $E(\sigma) \subseteq \operatorname{vert}(\tau)$. 
\end{proof}
\begin{prop}\label{prop:uniqueminperm}
Let $\sigma$ be a cyclically alternating permutation. If $\sigma^\prime$ minimizes inversion number in the cyclically alternating permutations $\tau$ satisfying $E(\sigma) \subseteq vert(\tau)$, then $\sigma = \sigma^\prime$.
\end{prop}

\begin{proof} 
    Follows from the proof of Proposition~\ref{prop:mininversion}.
\end{proof}

We quickly recall the notion of a shelling.
Let $\Lambda$ be the collection of maximal simplices in a pure simplicial complex of dimension $d$ with $|\Lambda| = s$. An ordering $\Delta^1, \Delta^2, \ldots, \Delta^s$ on the simplices in $\Lambda$ is a \emph{shelling order} if, for all $1 <r \leq s$,
$$
\bigcup_{i=1}^{r-1} (\Delta^i \cap \Delta^r)
$$
is a union of facets of $\Delta_r$.

\begin{thm}
\label{shelling}
Let $\sigma_1, \sigma_2, \ldots \sigma_{A_\circ(2n)}$ be an ordering of the cyclically alternating permutations so that $\operatorname{inv}(\sigma_i) \leq \operatorname{inv}(\sigma_j)$ when $i < j$. Then the associated order on $\Delta^{\sigma_1}, \Delta^{\sigma_2}, \ldots, \Delta^{\sigma_{A_\circ(2n)}}$ is a shelling order of the canonical triangulation of $\mathcal{O}({\mathcal{C}_{2n}})$.  
\end{thm}
\begin{proof}
We first show that it suffices to prove that if $\operatorname{inv}(\sigma) \geq \operatorname{inv}(\tau)$ (that is, $\Delta^\sigma$ comes after $\Delta^\tau$ in the shelling order), then $E(\sigma) \not \subset \operatorname{vert}(\tau)$. Towards this goal, let $\operatorname{inv}(\sigma) \geq \operatorname{inv}(\rho)$. By Proposition~\ref{prop:intinfacet}, $\Delta^\sigma \cap \Delta^\rho$ is a facet of $\Delta^\sigma$ if and only if $\rho$ swaps to $\sigma$ (since $\rho$ has less inversions and a swap increases inversions by $1$), which happens if and only if $\Delta^\sigma \cap \Delta^\rho = \Delta^{\sigma \setminus \{v_i\}}$ for $v_i \in E(\sigma)$. If $\Delta^\sigma \cap \Delta^\tau \not \subset \Delta^\sigma \cap \Delta^\rho$ for any $\rho$ with $\operatorname{inv}(\sigma) \geq \operatorname{inv}(\rho)$ and $\Delta^\sigma \cap \Delta^\rho$ a facet of $\Delta^\sigma$, then $E(\sigma) \subseteq \operatorname{vert}(\tau)$. The contrapositive of this last sentence says, if $E(\sigma) \not \subseteq \operatorname{vert}(\tau)$, then the given order is a shelling order. \par
If $\operatorname{inv}(\sigma) > \operatorname{inv}(\tau)$, then, since $\sigma$ minimizes the inversion number over all cyclically alternating permutations that contain $E(\sigma)$, by Proposition~\ref{prop:mininversion}, we have $E(\sigma) \not \subset \operatorname{vert}(\tau)$. If $\operatorname{inv}(\sigma) = \operatorname{inv}(\tau)$, then, by Proposition~\ref{prop:uniqueminperm}, $E(\sigma) \not \subset \operatorname{vert}(\tau)$, because $\sigma$ is the unique permutation that minimizes inversion number of all cyclically alternating permutations that contain $E(\sigma)$. 
\end{proof}

Let $P = ([n], \preceq)$ be a poset. The \emph{Jordan–H\"{o}lder set} of $P$ is defined as
\[
\operatorname{JH}(P) \coloneqq \{ \sigma = \sigma_1 \sigma_2 \cdots \sigma_n \in S_n \colon \mbox{ if } \sigma_i \preceq \sigma_j, \mbox{ then } i \leq j \mbox{ for all } i,j \in [n]\},
\]
that is, $\operatorname{JH}(P)$ is the set of all linear extensions of $P$.

The $h^*$-polynomial of $\mathcal{O}(P)$ can be computed in terms of a permutation statistic on the Jordan-H\"{o}lder set of permutations. We say that a poset $P = ([n], \preceq)$ is \emph{naturally labeled} if $i \prec j$ in $P$ implies $i < j$ as integers.

We say that $i$ is a \emph{descent} of a permutation $\sigma \in S_n$ if $\sigma_i > \sigma_{i+1}$. The \emph{descent set} of $\sigma$, denoted $\operatorname{Des}(\sigma)$, is the set of all descents of $\sigma$. The \emph{descent number} of $\sigma$, denoted $\operatorname{des}(\sigma)$, is the number of elements of $\operatorname{Des}(\sigma)$.

\begin{thm}[Theorem $6.3.11$ \cite{beck2018combinatorial}]
    \label{thm:JHstatistic}
    Let $P$ be a naturally labeled poset. The $h^{*}$-polynomial of $\mathcal{O}(P)$ is given by
    $$
        h^{*}({\mathcal{O}(P)},t) =  \sum_{\sigma \in \operatorname{JH}(P)} t^{\operatorname{des}(\sigma)}. 
    $$
\end{thm}

\begin{thm}\cite[Chapter 3]{beck2007computing}
    \label{hshelling}
    Let $\mathfrak{p}$ be a polytope with integer vertices. Let $\{ \Delta^1, \Delta^2, \ldots, \Delta^s \}$ be a unimodular triangulation of $\mathfrak{p}$ using no new vertices. If $\{ \Delta^1, \Delta^2, \ldots, \Delta^s \}$ is a shelling order, then $h^*_j$ is the number of $\Delta^i$ that are added along $j$ of their facets in this shelling. Equivalently,
    $$
    h^*(\mathfrak{p},t) = \sum_{i=1}^s t^{a_i},
    $$
    where $a_i = \# \{ k < i : \Delta^k \cap \Delta^i \text{ is a facet of } \Delta^i \}$.
\end{thm}

Now we are ready to prove the main theorem of our section.

\begin{thm}
\label{h-characterization}
    $h^*(\mathcal{O}({\mathcal{C}_{2n}}),t) = \sum_{\sigma \in \operatorname{CA_{2n}}} t^{\operatorname{cswap}(\sigma)}$.
\end{thm}
\begin{proof}
    Let $\sigma_1, \sigma_2, \ldots \sigma_{A_\circ(2n)}$ be a shelling order as described in Theorem~\ref{shelling}. Hence, by Proposition~\ref{prop:intinfacet} and Remark~\ref{rmk:cyclicswap}, each $\Delta^{\sigma_i}$ is added in the shelling along $\operatorname{cswap}(\sigma_i)$ facets. By Theorem~\ref{hshelling},
    $$
        h^*(\mathcal{O}({\mathcal{C}_{2n}}),t) = \sum_{\sigma \in \operatorname{CA_{2n}}} t^{\operatorname{cswap}(\sigma)}.
    $$
\end{proof}

Theorems~\ref{thm:JHstatistic} and~\ref{h-characterization} imply the following equidistribution:

\begin{cor}
    Let $\mathcal{C}_{2n}$ on $[2n]$ be naturally labeled. Then
    \[
        \sum_{\sigma \in \operatorname{CA_{2n}}} t^{\operatorname{cswap}(\sigma)} = \sum_{\sigma \in \operatorname{JH}(\mathcal{C}_{2n})} t^{\operatorname{des}(\sigma)}.
    \]
\end{cor}

\begin{cor}
\label{degree-h}
    $\deg (h^*(\mathcal{O}(\mathcal{C}_{2n}),t)) = 2n-2$.
\end{cor}
\begin{proof}
    Since the only permutation $\tau$ such that $\operatorname{des}(\tau)=2n-1$ is $\tau = (2n, 2n-1, \ldots, 1)$ and such permutation does not belong to the set $\operatorname{JH} (\crown_{2n})$, then $\deg (h^*(\mathcal{O}({\mathcal{C}_{2n}}),t)) \leq 2n-2$.

    Now, consider the CAP $\sigma = (n, 2n, n-1, 2n-1, \ldots, 1, n+1)$. Since $i$ is to the right of $i+1$ for every $i \in [2n] \setminus \{n, 2n\}$ and swapping $i$ and $i+1$ yields a CAP, then $\operatorname{cswap}(\sigma) = 2n-2$ and, hence, $\deg (h^*(\mathcal{O}({\mathcal{C}_{2n}}),t)) \geq 2n-2$.
\end{proof}

\begin{exam}
    For $n = 2$, we have $\operatorname{CA_{2n}} = \{ 1324, 1423, 2314, 2413 \}$. Hence
    
    \begin{table}[H]
    \centering
        \begin{tabular}{|| c || c | c | c | c ||} 
         \hline
         $\sigma$ & 1324 & 1423 & 2314 & 2413 \\
         \hline
         $\sigma^{-1}$ & 1324 & 1342 & 3124 & 3142 \\
         \hline
         $\operatorname{cswap}(\sigma)$ & 0 & 1 & 1 & 2 \\
         \hline
        \end{tabular}
    \end{table}
    and we have
    \[
        \sum_{\sigma \in \operatorname{CA_{4}}} t^{\operatorname{cswap}(\sigma)} = 1 + 2t + t^2.
    \]

    Now, consider the crown poset on $4$ elements, $\mathcal{C}_4$, with order relations defined by $1 \prec 3 \succ 2 \prec 4 \succ 1$.

    \begin{figure}[H]
        \centering
        \begin{tikzpicture}[line cap=round,line join=round,>=triangle 45,x=1cm,y=1cm]
            \clip(-9,6.2995524384280617) rectangle (-4,8.529946161913498);
            \draw [line width=1pt] (-8,7)-- (-7,8);
            \draw [line width=1pt] (-7,8)-- (-6,7);
            \draw [line width=1pt] (-6,7)-- (-5,8);
            \draw [line width=1pt] (-5,8)-- (-8,7);
            \draw (-8.020190989823679,7) node[anchor=north west] {1};
            \draw (-7.02493388454507,8.5) node[anchor=north west] {3};
            \draw (-6.016971369411841,7) node[anchor=north west] {2};
            \draw (-5.0,8.5) node[anchor=north west] {4};
            \begin{scriptsize}
                \draw [fill=ttqqqq] (-8,7) circle (1.5pt);
                \draw [fill=ttqqqq] (-7,8) circle (1.5pt);
                \draw [fill=ttqqqq] (-6,7) circle (1.5pt);
                \draw [fill=ttqqqq] (-5,8) circle (1.5pt);
            \end{scriptsize}
        \end{tikzpicture}
    \end{figure}

    Then $\operatorname{JH}(\mathcal{C}_{4}) = \{1234, 1243, 2134, 2143 \}$. Hence
    \begin{table}[H]
    \centering
        \begin{tabular}{|| c || c | c | c | c ||} 
         \hline
         $\sigma$ & 1234 & 1243 & 2134 & 2143 \\
         \hline
         $\operatorname{des}(\sigma)$ & 0 & 1 & 1 & 2 \\
         \hline
        \end{tabular}
    \end{table}
    and we have
    \[
        \sum_{\sigma \in \operatorname{JH}(\mathcal{C}_{2n})} t^{\operatorname{des}(\sigma)} = 1 + 2t + t^2.
    \]
\end{exam}

Recall that a lattice polytope $\mathfrak{p}$ is called \emph{reflexive} if there is an interior lattice point $x$ such that every facet has lattice distance one from $x$. In this case, $x$ is the unique interior lattice point of $\mathfrak{p}$. A lattice polytope is called \emph{Gorenstein} if there is some $r \in \mathbb{Z}_{\geq 1}$ such that $r\mathfrak{p}$ is reflexive. In this case, $r$ is called the \emph{index} of $\mathfrak{p}$. A poset is called \emph{pure} if all maximal chains have the same length.

In \cite{coons2019h}, it was shown that $\mathcal{O}(Z_n)$ is a Gorenstein polytope of index $3$ for every $n \geq 1$. We also show that $\mathcal{O}(\mathcal{C}_{2n})$ is a Gorenstein polytope of index $3$ for every $n \geq 1$ and present some consequences of this property.

\begin{prop} \label{prop:popgorenstein}
$\mathcal{O}({\mathcal{C}_{2n}})$ is Gorenstein for every $n \geq 1$.
\end{prop}
\begin{proof}
    By \cite[\textsection 3 d]{MR951198}, order polytopes of pure posets are Gorenstein.
\end{proof}
By \cite{stanley1986two}, order polytopes admit a unimodular triangulation. This combined with Proposition~\ref{prop:popgorenstein} and  \cite[Theorem $1$]{gorenstein2007}, gives us
\begin{cor}
\label{thm:symmhpop} 
    The coefficients of the $h^{*}$-polynomials of $\mathcal{O}({\mathcal{C}_{2n}})$ are symmetric and unimodal.
\end{cor}

\begin{table}[H]
\centering
\begin{tabular}{|| c || c ||} 
 \hline
 $n$ & $h^*(\mathcal{O}({\mathcal{C}_{2n}}))$ \\
 \hline
 \hline
 2 & $(1,2,1)$ \\
 \hline
 3 & $(1,11,24,11,1)$ \\
 \hline
 4 & $(1,38,263,484,263,38,1)$ \\
 \hline
 5 & $(1,112,1983,9684,16120,9684,16120,9684,1983,112,1)$ \\
 \hline
\end{tabular}
\caption{Table of $h^*$-vectors of $\mathcal{O}(\mathcal{C}_{2n})$ for $2 \leq n \leq 5$.}
\label{table-h-vector}
\end{table}

\begin{cor}
    The index of $\mathcal{O}({\mathcal{C}_{2n}})$ is $3$ and $x = (1,2,1,2, \ldots, 1,2)$ is the unique interior lattice point of $3\mathcal{O}({\mathcal{C}_{2n}})$.
\end{cor}

\begin{proof}
    Let $r$ be the index of $\mathcal{O}({\mathcal{C}_{2n}})$. By \cite[Theorem 1]{gorenstein2007},
    \begin{equation}
    \label{degree-h2}
        \deg(h^*(\mathcal{O}({\mathcal{C}_{2n}}),t)) = 2n-r+1.
    \end{equation}
    Equation \eqref{degree-h2} combined with Corollary~\ref{degree-h} imply that $r = 3$.

    For the second part, observe that the facet-defining inequalities of $3 \mathcal{O}({\mathcal{C}_{2n}})$ are
    \begin{gather*}
        0 \leq x_i \mbox{ if } i \mbox{ is odd}\\
        x_i \leq 3 \mbox{ if } i \mbox{ is even}\\
        x_{2i-1} \leq x_{2i} \mbox{ for every } i=1,2,\ldots,n \\
        x_{2i+1} \leq x_{2i} \mbox{ for every } i=0,1,\ldots,n-1 \\
        x_1 \leq x_{2n},
    \end{gather*}
    from which it is clear that $x = (1,2,1,2, \ldots, 1,2)$ is the unique interior lattice point of $3\mathcal{O}({\mathcal{C}_{2n}})$. Moreover, $x$ has lattice distance $1$ from each facet of $3\mathcal{O}({\mathcal{C}_{2n}})$.
\end{proof}

An interesting question about $h^*$-polynomials to be investigated is whether or not they have only real roots. The following conjecture was verified for $2 \leq n \leq 6$:

\begin{conj}
    $h^*(\mathcal{O}(\mathcal{C}_{2n}),t)$ is real-rooted for every $n \geq 2$.
\end{conj}

\subsection{The \texorpdfstring{$\gamma$}--nonnegativity of the \texorpdfstring{$h^*$}--vector}
\label{sec-gamma}
Let $g(t) = \sum_{i \geq 0} g_i t^i$ be a symmetric polynomial. Then, $g_i = g_{d-i}$ for some positive integer $d$, which we define as the \emph{symmetric degree} of $g$.

Symmetric polynomials with symmetric degree $d$ span a vector space of dimension $\lfloor d / 2 \rfloor + 1$. Hence, they can be written in the basis
\[
    \Gamma_d = \{ t^i (1+t)^{d-2i} \}_{i=0}^{\lfloor d / 2 \rfloor},
\]
that is,
\[
    g(t) = \sum_{0 \leq 2i \leq d} \gamma_i t^i(1+t)^{d-2i}.
\]
The vector $\gamma = (\gamma_0,\gamma_1, \ldots, \gamma_{\lfloor d / 2 \rfloor})$ is called the \emph{$\gamma$-vector} of $g$ and a topic of interest is the nonnegativity of the coefficients of the $\gamma$-vector of a polynomial.

The study of the nonnegativity of the $\gamma$-vector has gained significant attention due to its deep connections to combinatorics, and algebraic geometry. Initially introduced by Foata and Schützenberger \cite{MR272642} in the context of Eulerian polynomials and later developed further by various researchers such as Foata and Strehl \cite{MR347951,MR434831}, as well as Brändén \cite{MR2120105} and Gal \cite{MR2155722}, $\gamma$-nonnegativity emerged as a powerful concept that provides a novel approach to several important problems, particularly in combinatorics. One key area of application is in unimodality questions for symmetric polynomials. The nonnegativity of the $\gamma$-vector offers insights into the symmetric structure of a polynomial, revealing its unimodal nature and providing a potential approach for proving unimodality even for polynomials that are not real-rooted. For example, symmetric real-rooted polynomials in $\mathbb{R}_{\geq 0}$ have nonnegative $\gamma$-vectors \cite[Lemma 4.1]{MR2120105}. On the other hand, every $\gamma$-nonnegative polynomial (real-rooted or not) is symmetric and unimodal. This is particularly important in combinatorics, where polynomials such as Eulerian polynomials or $P$-Eulerian polynomials play a significant role in enumerating combinatorial objects (e.g., flag triangulations of spheres) and understanding combinatorial symmetries. Also, it is an interesting question finding a combinatorial interpretation of the coefficients of the $\gamma$-vector.

Let $P$ be a finite poset. A \emph{labeling} of $P$ is an injection $\omega \colon P \to \mathbb{Z}$. The pair $(P, \omega)$ is called \emph{naturally labeled} if $x \leq_P y$ implies $\omega(x) \leq \omega(y)$. In this case, $\omega$ is called a \emph{natural labeling}. Let $(\Tilde{P}, \Tilde{\omega})$ be any \emph{canonically labeled poset} \cite[Page 9]{MR2388387} obtained from $(P, \omega)$ by adjoining a greatest element.

By \cite[Theorem 4.2]{MR2120105}, $h^*(\mathcal{O}(P),t)$ has nonnegative expansion in the basis $\Gamma_{|P|-r-1}$ for any sign-graded poset of rank $r$, and, by \cite[Theorem 6.4]{MR2388387}, the coefficients of their $\gamma$-vectors have a combinatorial interpretation. Since naturally graded posets are also sign-graded, we deduce the following:

\begin{thm}\label{thm:gamma_nonneg}
    For each $n \geq 1$, the polynomial $h^*(\mathcal{O}(\mathcal{C}_{2n}),t)$ has a nonnegative expansion in the basis $\Gamma_{2n-2}$, that is, there are nonnegative integers $\gamma_{2n,j}$ such that
    \[
        h^*(\mathcal{O}(\mathcal{C}_{2n}),t) = \sum_{0 \leq j \leq n-1} \gamma_{2n,j} t^j (1+t)^{2n-2-2j}
    \]
    and
    \[
        \gamma_{2n,j} = 2^{-2n+2+2j} |\{ \pi \in \operatorname{JH}(\Tilde{\mathcal{C}}_{2n}, \Tilde{\omega}) \colon \peak(\pi) = j+1 \}|,
    \]
    where $\omega$ is any natural labeling of $\mathcal{C}_{2n}$ and $\peak(\pi) = |\{ i \colon a_{i-1} < a_i > a_{i+1} \}|$.
\end{thm}

\begin{exam}
    Recall that $h^*(\mathcal{O}({\mathcal{C}_{4}}),t) = 1 + 2t + t^2$. Then,
    \[
        h^*(\mathcal{O}(\mathcal{C}_{2n}),t) = 1 \cdot t^0 (1+t)^2 + 0 \cdot t^1 (1+t)^0,
    \]
    from which follows that $\gamma_{4,0} = 1$ and $\gamma_{4,1} = 0$. By Theorem \ref{thm:gamma_nonneg}, we now verify that
    \[
        |\{ \pi \in \operatorname{JH}(\Tilde{\mathcal{C}}_{4}, \Tilde{\omega}) \colon \peak(\pi) = 1 \}| = 4
    \]
    and
    \[
        |\{ \pi \in \operatorname{JH}(\Tilde{\mathcal{C}}_{4}, \Tilde{\omega}) \colon \peak(\pi) = 2 \}| = 0.
    \]
    
    In fact, a canonical labeling $\Tilde{\omega}$ of $\Tilde{\mathcal{C}}_{4}$ is given in Figure \ref{fig:new-example}. Hence,
    \[
        \operatorname{JH}(\Tilde{\mathcal{C}}_{4}, \Tilde{\omega}) = \{ 23451, 23541, 32451, 32541 \}.
    \]
    Since $\peak(23451) = \peak(23541) = \peak(32451) = \peak(32541) = 1$, the claim follows.
    
    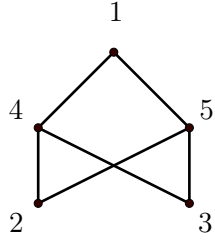
\begin{figure}[H]
        \centering
        \begin{tikzpicture}[line cap=round,line join=round,>=triangle 45,x=1cm,y=1cm]
            \clip(-9,6.2995524384280617) rectangle (-4,10);
            \draw [line width=1pt] (-7,7)-- (-7,8);
            \draw [line width=1pt] (-7,8)-- (-5,7);
            \draw [line width=1pt] (-5,7)-- (-5,8);
            \draw [line width=1pt] (-5,8)-- (-7,7);
            \draw [line width=1pt] (-7,8)-- (-6,9);
            \draw [line width=1pt] (-5,8)-- (-6,9);
            \draw (-7.520190989823679,7) node[anchor=north west] {2};
            \draw (-7.52493388454507,8.5) node[anchor=north west] {4};
            \draw (-5.016971369411841,7) node[anchor=north west] {3};
            \draw (-5.0,8.5) node[anchor=north west] {5};
            \draw (-6.2,9.8) node[anchor=north west] {1};
            \begin{scriptsize}
                \draw [fill=ttqqqq] (-7,7) circle (1.5pt);
                \draw [fill=ttqqqq] (-7,8) circle (1.5pt);
                \draw [fill=ttqqqq] (-5,7) circle (1.5pt);
                \draw [fill=ttqqqq] (-5,8) circle (1.5pt);
                \draw [fill=ttqqqq] (-6,9) circle (1.5pt);
            \end{scriptsize}
        \end{tikzpicture}
        \caption{A canonical labeling of $\Tilde{\mathcal{C}}_{4}$.}
        \label{fig:new-example}
    \end{figure}
\end{exam}

\section{Open problems and further directions}
\label{sec-open-further}

We conclude this work with some open questions and further directions. One of the reasons why we studied the order polytope of the crown poset is given in this section. Recall that the \emph{ordinal sum} of two posets $P$ and $Q$, denoted by $P \oplus Q$, is the poset on $P \cup Q$ such that $x \leq y$ in $P \oplus Q$ if $x \leq y$ in $P$, or if $x \leq y$ in $Q$, or if $x \in P$ and $y \in Q$. This implies

\begin{prop}\cite[Lemma 7.3]{MR4108204}
    Let $P$ and $Q$ be finite posets. Then
    \[
        \mathcal{O}(P \oplus Q) = \conv(\{1\}^Q \times \mathcal{O}(P) \cup \mathcal{O}(Q) \times \{0\}^P) \subseteq \mathbb{R}^{P \sqcup Q}.
    \]
\end{prop}

In particular, the above proposition shows that adding a minimum or a maximum element to $P$ corresponds to taking a pyramid over the original order polytope. Hence, if $z$ is not an element of $P$, then
\begin{equation}
\label{eq-face-lattice-POP}
    f(\mathcal{O}(z \oplus P),t) = f(\mathcal{O}(P \oplus z),t) = (1+t) f(\mathcal{O}(P),t).
\end{equation}

Let $P_n$ be a polygon with $n$ vertices. Then its face lattice $\mathcal{L}(P_n)$ is isomorphic to $\hat{0} \oplus \mathcal{C}_{2n} \oplus \hat{1}$. Then, by Equation \eqref{eq-face-lattice-POP}, the $f$-polynomial of $\mathcal{O}(\mathcal{L}(P_n))$ is given by
\begin{equation}
\label{polygon-f}
    f(\mathcal{O}(\mathcal{L}(P_n)),t) = (1+t)^2 f(\mathcal{O}(\mathcal{C}_{2n}),t).
\end{equation}
By \cite[Theorem 1.2, Chapter 8]{log1968}, the product of log-concave polynomials is also log-concave. Hence, Theorem~\ref{conj-f-vector} implies
\begin{cor}
\label{conj-f-vector2}
    For all $n \geq 2$, the entries of the $f$-vector of $\mathcal{O}(\mathcal{L}(P_n))$ form a log-concave sequence.
\end{cor}
Moreover, Corollary \ref{cor:asymp-normal} and Equation \ref{polygon-f} imply
\begin{cor}
    The coefficients of the sequence of polynomials $\{ f(\mathcal{O}(\mathcal{L}(P_n)),t) \}_{n \geq 2}$ are asymptotically normal.
\end{cor}

The order polynomial behaves well under the ordinal sum of a poset $P$ and an element $z$ not in $P$:

\begin{prop}
    Let $P$ be a finite poset and $z$ be an element not in the ground set of $P$. Then
    \[
        \Omega_{z \oplus P}(t) = \Omega_{P \oplus z}(t) = \sum_{k=0}^{|P|} [e_k(P) + e_{k-1}(P)] \binom{t}{k},
    \]
    where $e_k(P)$ is the number of order preserving surjections $\alpha \colon P \to [k]$. In particular,
    \[
        L(z \oplus P,t) = L(P \oplus z,t) = \sum_{k=0}^{|P|} [e_k(P) + e_{k-1}(P)] \binom{t+1}{k}.
    \]
\end{prop}
\begin{proof}
    Consider the poset $P \oplus z$ and let $E_k (P \oplus z)$ be the set of order-preserving surjections
    \begin{equation}
        \alpha \colon P \oplus z \to [k].
    \end{equation}
    Since $z$ is a maximum element of $P \oplus z$, $\alpha (z) = k$ for every $\alpha \in E_k (P \oplus z)$.

    The set $E_k (P \oplus z)$ can be written as the disjoint union $E_k^1 (P \oplus z) \sqcup E_k^2 (P \oplus z)$, where
    \[
        E_k^1 (P \oplus z) = \{ \alpha \in  E_k (P \oplus z) \colon |\alpha^{-1}(k)| = 1 \}
    \]
    and
    \[
        E_k^2 (P \oplus z) = \{ \alpha \in  E_k (P \oplus z) \colon |\alpha^{-1}(k)| > 1 \}.
    \]
    Since $E_k^1 (P \oplus z)$ is in bijection with $E_{k-1}(P)$ and $E_k^2 (P \oplus z)$ is in bijection with $E_k (P)$, the result follows.
\end{proof}

Unfortunately, the ordinal sum does not preserve the nonnegativity of the coefficients of the order polynomial. For example, if $A(4)$ is the antichain on $4$ elements, then $[t]\Omega_{A(4)}(t) = 0$, but $[t]\Omega_{z \oplus A(4)}(t) = -1/30$.

With the help of Sage \cite{sagemath}, it was verified that $[t]\Omega_{\mathcal{L}(P_7)}(t) = -3/1430$. However, we could not find an example of polygon $P_n$ such that the Ehrhart polynomial of $\mathcal{O}(\mathcal{L}(P_n))$ has negative coefficients. The following question is still open:

\begin{question}
    Is $\mathcal{O}(\mathcal{L}(P_n))$ Ehrhart-nonnegative for every natural number $n \geq 3$?
\end{question}

The $h^*$-polynomial of the order polytope of the ordinal sum of two finite posets, $P$ and $Q$, is also well-known:
\begin{thm}\cite[Proposition 7.4]{MR4108204}
\[
    h^*(\mathcal{O}(P \oplus Q),t) = h^*(\mathcal{O}(P),t) \cdot h^*(\mathcal{O}(Q),t).
\]
\end{thm}
In particular, if $z$ is not an element of $P$, then
\[
    h^*(\mathcal{O}(P),t) = h^*(\mathcal{O}(z \oplus P),t) = h^*(\mathcal{O} (P \oplus z),t),
\]
which implies
\begin{equation}
\label{polygon-h}
    h^*(\mathcal{O}(\mathcal{L}(P_n)),t) = h^*(\mathcal{O}(\mathcal{C}_{2n}),t).
\end{equation}

If $P$ is a poset isomorphic to the face lattice of a polytope, we call $\mathcal{O}(P)$ a \emph{polytope order polytope}, or \emph{POP}, for short. For polytopes of dimensions $0$ (points) and $1$ (line segments), their POPs are completely understood. On the other hand, Equations \eqref{polygon-f} and \eqref{polygon-h} describe the $f$-polynomial and the $h^*$-polynomial of POPs of $2$-dimensional polytopes, so a natural question is:
\begin{question}
    How can one describe the $f$-polynomial, the Ehrhart polynomial, and the $h^*$-polynomial of POPs of $n$-dimensional polytopes for $n \geq 3$?
\end{question}

Two natural constructions arise from the $2$-dimensional case:
\begin{enumerate}[(i)]
    \item pyramids over polygons, and
    \item prisms over polygons.
\end{enumerate}
In fact, the face lattice of a pyramid over a polygon $P_n$ is isomorphic to $\mathcal{L}(P_n) \times \mathcal{P}_2$, where $\mathcal{P}_2$ is the chain poset on $2$ elements, and the face lattice of a prism over $P_n$ is isomorphic to
\begin{equation*}
    \hat{0} \oplus [(\mathcal{L}(P_n) \setminus \{ \hat{0} \}) \times (\mathcal{L}(S) \setminus \{ \hat{0} \})],
\end{equation*}
where $S$ is a line segment. Hence, a first step to studying POPs of $3$-dimensional polytopes could be investigating the order polytope of the cartesian product of two posets. This sounds like a hard problem, but, since $\mathcal{P}_2$ and $\mathcal{L}(S)$ are posets with few elements and order relations, the following question could perhaps be more feasible:
\begin{question}
\label{question-cartesian}
    Given a finite poset $P$, how can one describe the $f$-vector and the $h^*$-vector of $\mathcal{O}(P \times \mathcal{P}_2)$ and $\mathcal{O}(P \times (\mathcal{L}(S) \setminus \{ \hat{0} \}))$?
\end{question}

The particular case of taking the cartesian product between a finite poset $P$ and the $2$-chain poset $\mathcal{P}_2$ is interesting because of the following: recall that the \emph{Boolean lattice} $B_n$ is the poset on $2^{[n]}$ such that $A \preceq B$ in $B_n$ if $A \subseteq B$. It is not hard to verify that
\[
    B_n \cong \overbrace{\mathcal{P}_2 \times \mathcal{P}_2 \times \cdots \times \mathcal{P}_2}^{n \mbox{ times}}.
\]
An $n$-simplex is an $n$-dimensional polytope that is the convex hull of $n+1$ affinely independent points, and its face lattice is isomorphic to $B_{n+1}$, that is, answering Question~\ref{question-cartesian} gives information about POPs of simplices. In particular, if the real-rootedness of $h^*(\mathcal{O}(P),t)$ is preserved under the cartesian product, that is, if the real-rootedness of $h^*(\mathcal{O}(P),t)$ implies the real-rootedness of $h^*(\mathcal{O}(P \times \mathcal{P}_2),t)$, where $P$ is a finite poset, then the following conjecture would be confirmed to be true:
\begin{conj}
    $h^*(\mathcal{O}(B_n),t)$ is real-rooted for every natural number $n$.
\end{conj}

\bibliographystyle{plain}
\bibliography{biblio}
\end{document}